\documentclass[11pt]{article}
\usepackage{amsmath}
\usepackage{amsfonts}
\usepackage{amssymb}
\usepackage{mathrsfs}
\usepackage{amsthm}
\usepackage{palatino}
\usepackage[margin=2.5cm, vmargin={1.5cm}]{geometry}

\usepackage{dcolumn}

\usepackage{times}
\usepackage{graphicx}
\usepackage{xcolor}

\usepackage{pstricks}
\usepackage{pst-plot}

\renewcommand\footnotemark{}

\begin{document}

\title{A generalization of the Riemann-Siegel formula}

\author{
  Cormac ~O'Sullivan\footnote{{\it Date:} April 18, 2019.
\newline \indent \ \ \
  {\it 2010 Mathematics Subject Classification:} 11M06.
  \newline \indent \ \ \
Support for this project was provided by a PSC-CUNY Award, jointly funded by The Professional Staff Congress and The City
\newline \indent \ \ \
University of New York.}
  }

\date{}

\maketitle

\def\s#1#2{\langle \,#1 , #2 \,\rangle}

\def\H{{\mathbf{H}}}
\def\F{{\frak F}}
\def\C{{\mathbb C}}
\def\R{{\mathbb R}}
\def\Z{{\mathbb Z}}
\def\Q{{\mathbb Q}}
\def\N{{\mathbb N}}
\def\G{{\Gamma}}
\def\GH{{\G \backslash \H}}
\def\g{{\gamma}}
\def\L{{\Lambda}}
\def\ee{{\varepsilon}}
\def\K{{\mathcal K}}
\def\Re{\mathrm{Re}}
\def\Im{\mathrm{Im}}
\def\PSL{\mathrm{PSL}}
\def\SL{\mathrm{SL}}
\def\Vol{\operatorname{Vol}}
\def\lqs{\leqslant}
\def\gqs{\geqslant}
\def\sgn{\operatorname{sgn}}
\def\res{\operatornamewithlimits{Res}}
\def\li{\operatorname{Li_2}}
\def\lip{\operatorname{Li}'_2}
\def\pl{\operatorname{Li}}

\def\clp{\operatorname{Cl}'_2}
\def\clpp{\operatorname{Cl}''_2}
\def\farey{\mathscr F}

\newcommand{\stira}[2]{{\genfrac{[}{]}{0pt}{}{#1}{#2}}}
\newcommand{\stirb}[2]{{\genfrac{\{}{\}}{0pt}{}{#1}{#2}}}
\newcommand{\norm}[1]{\left\lVert #1 \right\rVert}

\newcommand{\e}{\eqref}


\newtheorem{theorem}{Theorem}[section]
\newtheorem{lemma}[theorem]{Lemma}
\newtheorem{prop}[theorem]{Proposition}
\newtheorem{conj}[theorem]{Conjecture}
\newtheorem{cor}[theorem]{Corollary}

\newcounter{coundef}
\newtheorem{adef}[coundef]{Definition}

\newcounter{counrem}
\newtheorem{remark}[counrem]{Remark}

\renewcommand{\labelenumi}{(\roman{enumi})}
\newcommand{\spr}[2]{\sideset{}{_{#2}^{-1}}{\textstyle \prod}({#1})}
\newcommand{\spn}[2]{\sideset{}{_{#2}}{\textstyle \prod}({#1})}

\numberwithin{equation}{section}

\let\originalleft\left
\let\originalright\right
\renewcommand{\left}{\mathopen{}\mathclose\bgroup\originalleft}
\renewcommand{\right}{\aftergroup\egroup\originalright}

\bibliographystyle{alpha}

\begin{abstract}
The celebrated Riemann-Siegel formula  compares the Riemann zeta function on the critical line with its partial sums, expressing the difference between them as an expansion in terms of decreasing powers of the imaginary variable $t$. Siegel anticipated that this formula could be generalized to include the Hardy-Littlewood approximate functional equation, valid in any vertical strip. We give this generalization for the first time. The asymptotics contain Mordell integrals and an interesting new family of polynomials.
\end{abstract}

\section{Introduction}
\subsection{The approximate functional equation}
The functional equation for the Riemann zeta function $\zeta(s)$ may be written as
\begin{equation} \label{fe}
   \zeta(s)=\chi(s) \zeta(1-s)
\end{equation}
for
\begin{equation*}
  \chi(s):=\pi^{s-1/2}\G\left(\frac{1-s}{2}\right)\G\left(\frac{s}{2}\right)^{-1}.
\end{equation*}
Hardy and Littlewood  gave the following approximation for $\zeta(s)$ in \cite{HL23}. The notation $s=\sigma+ i t$ for the real and imaginary parts of $s$ is assumed for here on.

\begin{theorem}[The  Hardy-Littlewood approximate functional equation] \label{hl0}
Let $I \subset \R$ be a finite interval. Let $s$ be any complex number in the vertical strip described by $\sigma \in I$ and $t\gqs 2\pi$. Then for all $\alpha,\beta \in \R_{\gqs 1}$ with  $t=2\pi \alpha \beta$,  we have
\begin{equation} \label{hali0}
  \zeta(s)= \sum_{n \lqs \alpha} \frac {1}{n^{s}}+\chi(s)\sum_{n\lqs \beta} \frac {1}{n^{1-s}}
  +O\left(\alpha^{-\sigma}+t^{1/2-\sigma}\beta^{\sigma-1} \right)
\end{equation}
where the implied constant depends only on  $I$.
\end{theorem}

The sums in \e{hali0}, and similar sums below, are over all positive integers $n$ satisfying the given conditions.
 Our use of the big $O$ notation is as in \cite{HL23} and \cite[p. 7]{IwKo}, for example. Writing $f(x)=O(g(x))$ (or equivalently  $f(x) \ll g(x)$) means that, for an explicitly specified range $X$, there is an implied constant $C>0$ so that $|f(x)| \lqs C\cdot g(x)$ for all $x\in X$. Similarly, the notation extends to functions of more than one variable. With this convention, the implied constant in Theorem \ref{hl0} may depend on $I$, but gives a bound that is valid for all $s$, $\alpha$ and $\beta$ satisfying the given conditions. In this way, for instance, the qualifier ``as $t\to \infty$'' is not needed for \e{hali0}.

Hardy and Littlewood used Theorem \ref{hl0} to estimate the second and fourth moments of $\zeta$ on the critical line with real part $1/2$. See for example \cite[Chapter 7]{Ti86}, \cite[Chapters 5,8,15]{Iv85} for more on the important  general moment problem and \cite{So09} for descriptions of more recent results and conjectures.

If $\alpha$ and $\beta$ are similar in size then the error in \e{hali0} is about $O(t^{-\sigma/2})$ and hence small for $t$ large and $\sigma$ positive. Thus, $\zeta_\alpha(s):=\sum_{n \lqs \alpha} n^{-s}+\chi(s)\sum_{n\lqs \alpha} n^{-1+s}$ gives a good approximation to $\zeta(s)$ when $t$ is close to $2\pi \alpha^2$.
For  positive  fixed $\alpha$, the function $\zeta_\alpha(s)$ is interesting in its own right. It is shown in \cite[Thm. 1.5]{gon13} that, in a natural sense, $100\%$ of its zeros  are simple and lie on the critical line.

Following Siegel in \cite[Eq. (36)]{Si32}, we define  $\vartheta(s):=(i/2)\log \chi(s)$ for $s\in \C$ with $s$ outside the intervals $(-\infty,0]\cup [1,\infty)$. The requirement $\vartheta(1/2)=0$ specifies the branch uniquely; see Section \ref{rsf} for more details. (We do not use the common notation $\vartheta(t)$ for $(i/2)\log \chi(1/2+i t)$ as it is not well suited for working off the critical line.)  As a consequence of Corollary \ref{rsthe} we have, for example,
\begin{equation*}
  i \vartheta(s) = \left( \frac s2-\frac 14\right) \log \frac{|t|}{2\pi} -\frac{i t}2 -\sgn(t)\frac{i \pi}8 +O\left( \frac 1{|t|}\right)
\end{equation*}
for all $t\neq 0$ with the implied constant depending on $\sigma$.  
Also $\vartheta(s)$ satisfies the relations
\begin{equation} \label{thsym}
  \vartheta(1-s) = -\vartheta(s), \qquad \overline{\vartheta(s)} = -\vartheta(\overline{s}).
\end{equation}
Hence, with $\chi(s)=e^{-2 i \vartheta(s)}$, we may write the functional equation \e{fe} in the symmetric form
\begin{equation} \label{fesym}
   e^{ i \vartheta(s)}\zeta(s)=  e^{ i \vartheta(1-s)} \zeta(1-s).
\end{equation}
It follows that Theorem \ref{hl0} has the  equivalent restatement:

\begin{theorem}[The  Hardy-Littlewood approximate functional equation, symmetric version] \label{hl}
Let $I \subset \R$ be a finite interval. Let $s$ be any complex number in the vertical strip described by $\sigma \in I$ and $t\gqs 2\pi$. Then for all $\alpha,\beta \in \R_{\gqs 1}$ with  $t=2\pi \alpha \beta$,  we have
\begin{equation} \label{hali}
  e^{i\vartheta(s)}\zeta(s)= e^{i\vartheta(s)}\sum_{n \lqs \alpha} \frac {1}{n^{s}}+e^{i\vartheta(1-s)}\sum_{n\lqs \beta} \frac {1}{n^{1-s}}
  +O\left(\frac{\lambda^{1/2-\sigma}\left(\lambda^{1/2}+\lambda^{-1/2}\right)}{t^{1/4}} \right)
\end{equation}
where $\lambda :=\sqrt{\alpha/\beta}$ and the implied constant depends only on  $I$.
\end{theorem}

We used that
\begin{equation}\label{cle}
  t=2\pi \alpha \beta \quad \text{and} \quad \lambda =\sqrt{\frac\alpha\beta} \quad \implies  \quad \alpha = \lambda \sqrt{\frac{t}{2\pi}} \quad \text{and} \quad  \beta = \frac 1\lambda \sqrt{\frac{t}{2\pi}}.
\end{equation}

\subsection{The Riemann-Siegel formula}
The Riemann-Siegel formula is one of the key results in the theory of the zeta function. It gives a detailed description of what is happening inside the error terms 
in Theorems \ref{hl0} and \ref{hl}, at least in the case where the lengths of the partial sums are the same: $\alpha=\beta$ and $\lambda=1$. Of course Riemann's researches predate those of Hardy and Littlewood by many years. The formula was discovered by Siegel  in Riemann's  unpublished notes  and appeared in \cite{Si32}.  Siegel's classic paper has been recently translated in \cite{BS} and we use their page numbering, corresponding to the version of the paper appearing in his collected works.

Most major computations verifying the Riemann hypothesis are based on the Riemann-Siegel formula; see for example \cite{Br},  \cite{OS88}, \cite{Go04}, \cite{BH18} and the contained references. It also appears in theoretical work where precise knowledge of $\zeta(s)$ on the critical line is required, such as \cite{Fe05,PT15}.

Let
\begin{equation} \label{psiw}
  \Psi(u):=\frac{\cos(\pi (u^2/2-u-1/8 ))}{\cos(\pi u)},
\end{equation}
which may be seen to be an entire function. The following result in a slightly different notation is given in \cite[Eqns. (32), (33)]{Si32}.

\begin{theorem}[The Riemann-Siegel formula for $\sigma \in I$] \label{hat0} Let $I \subset \R$ be a finite interval and let $s$ be any complex number in the vertical strip described by $\sigma \in I$ and $t\gqs 2\pi$. Suppose $\alpha:=\sqrt{t/(2\pi)}$ has fractional part $a\in [0,1)$. Then we have
\begin{multline} \label{fab0}
  \zeta(s) = \sum_{n \lqs \alpha} \frac 1{n^{s}}+\chi(s) \sum_{n\lqs \alpha} \frac{1}{n^{1-s}}
  +
   \frac {(-1)^{\lfloor \alpha \rfloor} (2\pi)^s e^{\pi i s/2}}{\G(s)(e^{2\pi i s}-1)}
\exp\left( \left( \frac s2-\frac 14\right) \log \frac t{2\pi} -\frac{i t}2 -\frac{i \pi}8\right)
   \\
  \times \left(\frac{2\pi}t\right)^{1/4} \sum_{k=0}^{N-1} a_k(s)
  \sum_{r=0}^{\lfloor k/2 \rfloor}
\frac{i^{r-k} \cdot k!}{r! (k-2r)! }  \frac{\Psi^{(k-2r)}(2a)}{4^r (2\pi)^{k/2-r}}
  +O\left( \frac{1}{t^{N/6+\sigma/2}}\right).
\end{multline}
 The implied constant depends only on $I$ and $N \in \Z_{\gqs 0}$.
\end{theorem}

Siegel was also able to bound the error's dependence on $N$ in \e{fab0} for $t$ large enough.
 The functions $a_k(s)$ may be defined recursively by $a_{-2}(s)=a_{-1}(s)=0$, $a_0(s)=1$ and
\begin{equation} \label{akrec}
  (k+1)\sqrt{t} \cdot a_{k+1}(s) = -(k+1-\sigma)a_k(s) + i \cdot a_{k-2}(s) \qquad (k \in \Z_{\gqs 0}).
\end{equation}

 Theorem \ref{hat0} is in fact an intermediate result. In what we may call the completed form, the terms  on the right of \e{fab0} are  expanded in decreasing powers of $t$. It is also useful to make things symmetric  by multiplying by $e^{i\vartheta(s)}$. This was Riemann's goal, as shown in \cite[Eqns. (44), (45)]{Si32}, and the following theorem is stated in \cite[p. 143]{Si43}.

\begin{theorem}[The Riemann-Siegel formula: completed, symmetric version for $\sigma =1/2$] \label{rsfthm}
Let $a\in [0,1)$ be the fractional part of $\alpha:=\sqrt{t/(2\pi)}$.
For any $N\in \Z_{\gqs 0}$, there exist explicit functions of $a$ alone, $C_0(a)$, $C_1(a)$, $C_2(a), \dots$, such that
\begin{equation} \label{rt}
 e^{i\vartheta(s)}\zeta(s) = e^{i\vartheta(s)} \sum_{n\lqs \alpha}\frac{ 1}{n^{s}}+
   e^{i\vartheta(1-s)} \sum_{n\lqs \alpha}\frac{1}{n^{1-s}}
   +(-1)^{\lfloor \alpha \rfloor +1}
   \left( \frac{2\pi}{t}\right)^{1/4} \sum_{m=0}^{N-1} \frac{C_m(a)}{t^{m/2}} + O\left(\frac{1}{t^{N/2+1/4}}\right)
\end{equation}
for all $s=1/2+it$ with $t\gqs 2\pi$. The implied constant in \e{rt} depends only on $N$.
\end{theorem}

Riemann computed the initial terms in \e{rt} exactly and the first four are
\begin{subequations} \label{riem}
\begin{align}
  C_0(a) & = \Psi(2a), \\
  C_1(a) & = -\frac{1}{3}(2 \pi)^{-3/2} \Psi^{(3)}(2a),\\
  C_2(a) & = \frac{1}{18}(2\pi)^{-3} \Psi^{(6)}(2a) + \frac{1}{4}(2\pi)^{-1} \Psi^{(2)}(2a),\\
  C_3(a) & = -\frac{1}{162} (2 \pi)^{-9/2} \Psi^{(9)}(2a) -\frac{2}{15} (2\pi)^{-5/2} \Psi^{(5)}(2a)
-\frac{1}{8} (2\pi)^{-1/2} \Psi^{(1)}(2a).
\end{align}
\end{subequations}
Siegel proved in \cite[p. 293]{Si32} that only derivatives $\Psi^{(r)}(2a)$ for $r\equiv 3m \bmod 4$ appear in $C_m(a)$.

The left side of \e{rt} for $s=1/2+it$ is $Z(t)$,
  Hardy's $Z$ function, and the sums over $n$ may be combined so that  \e{rt} becomes
\begin{equation} \label{loh2}
  Z(t)= 2\sum_{n \lqs \alpha} \frac {\cos(\vartheta(1/2+i t)-t\log n)}{n^{1/2}}
   +(-1)^{\lfloor \alpha \rfloor +1}
   \left( \frac{2\pi}{t}\right)^{1/4} \sum_{m=0}^{N-1} \frac{C_m(a)}{t^{m/2}} + O\left(\frac{1}{t^{N/2+1/4}}\right).
\end{equation}
For $t\in \R$ we have that $\vartheta(1/2+i t)$ and $Z(t)$ are real. Therefore zeros of $\zeta(s)$ on the critical line correspond to zeros of $Z(t)$.  Riemann was able to find the first zeros in the critical strip, as Edwards recounts in \cite[Sect. 7.6]{Ed74}, although it is not clear if he was using this formula in his calculations. Subsequent work using \e{loh2} has verified the Riemann hypothesis up to a large height. In these applications  it is important to give exact bounds on the error in \e{rt}, \e{loh2}. This has been achieved, for example, by Titchmarsh for $N=1$ \cite[p. 390]{Ti86} and  Gabcke for all $N \lqs 10$ \cite[Eq. (8)]{Ga79}. In this paper we will not give explicit error bounds.

 Table \ref{tgz} shows an example of the approximations of Theorem \ref{rsfthm} and \e{loh2} to $Z(t)$ for $t=2\pi$ and different values of $N$. We will see that the accuracy is better for larger values of $t$.

\begin{table}[ht]
\centering
\begin{tabular}{cccc|c}
$N=0$ & $N=1$ & $N=3$ &  $N=6$ & $Z(2\pi)$  \\ \hline
$-1.85029$ & $-0.926411$ & $-0.955739$ &  $-0.956017$  &
$-0.956029$
\end{tabular}
\caption{The approximations of Theorem \ref{rsfthm} to $Z(2\pi)$.} \label{tgz}
\end{table}

Riemann and Siegel gave recursive procedures for calculating the coefficients $C_m(a)$. Gabcke in \cite{Ga79} provided a different method of proof of  Theorem \ref{rsfthm} and  a new recursion for the coefficients $C_m(a)$. The  starting point in \cite{Ga79} is another unpublished formula of Riemann appearing in \cite{Si32}, namely
\begin{equation}\label{bessel}
  \zeta(s)=\int_{0 \swarrow 1} \frac{z^{-s} e^{\pi i z^2}}{e^{\pi i z}- e^{-\pi i z}}\, dz +\chi(s)\int_{0 \searrow 1} \frac{z^{-s} e^{-\pi i z^2}}{e^{\pi i z}- e^{-\pi i z}}\, dz.
\end{equation}
The paths of integration are lines that pass through the interval $(0,1)$ in the indicated direction. We also mention an interesting formal derivation of the $C_m(a)$ by Berry in \cite{Be95}.

  As Theorem \ref{hat0} is valid in any vertical strip, it is natural to seek an extension of Theorem \ref{rsfthm} that is also valid off the critical line. Arias de Reyna in \cite{Ar11} gave a Riemann-Siegel formula for the left integral on the right side of \e{bessel} that holds in any vertical strip. With the same assumptions as Theorem \ref{hat0}, his result may be stated for $\zeta(s)$ as
\begin{multline} \label{fab0adr}
  \zeta(s) = \sum_{n \lqs \alpha} \frac 1{n^{s}}+\chi(s) \sum_{n\lqs \alpha} \frac{1}{n^{1-s}}
  +
   (-1)^{\lfloor \alpha \rfloor+1}  \sum_{k=0}^{N-1} \frac{1}{t^{k/2}}\\
\times \left[ t^{-\sigma/2}U(t) \cdot D_k(a,\sigma)+\chi(s)  t^{(\sigma-1)/2}\overline{U(t) \cdot D_k(a,1-\sigma)}\right]
  +O\left( \frac{t^{-\sigma/2}+t^{(\sigma-1)/2} }{t^{N/2}}\right)
\end{multline}
for $U(t):=\exp\left( -\frac{it}2 \log \frac t{2\pi} +\frac{i t}2 +\frac{i \pi}8\right)$ and certain recursively defined functions $D_k(a,\sigma)$ as described in \cite[Sect. 2]{Ar11}. The implied constant depends on $I$ and $N$ and is bounded explicitly in \cite[Thm. 4.2]{Ar11}.

\subsection{Main results}
In this paper we generalize the Riemann-Siegel formula to the case where the lengths of the partial sums, $\alpha$ and $\beta$, may be different, as in the results of Hardy and Littlewood.
Siegel himself, in \cite[Sect. 4]{Si32}, suggested this should be possible without much difficulty and even gave the function that would be needed in place of $\Psi$.
For $u,\tau \in \C$ with $\Re(\tau)>0$, it is
\begin{equation} \label{mord}
  \Upsilon(u;\tau):=\int_{0 \nwarrow 1} \frac{e^{-\pi i \tau z^2 +2\pi i u z}}{e^{2\pi i z}-1} \, dz
\end{equation}
where the path of integration   is again a  line   crossing the  interval $(0,1)$  in the indicated direction. It is straightforward to see that the integral converges rapidly for $\Re(\tau)>0$ and is independent of the choice of line.
Siegel used  $\Phi(-\tau,u)$ for $\Upsilon(u;\tau)$, so the new notation should avoid confusion.

We need a suitably normalized version of Siegel's function:
\begin{equation} \label{fdef}
  G(u;\tau) := \tau^{1/4}\exp\left( -\frac{\pi i u^2}{2}+\frac{\pi i}{8}\right)  \Upsilon \left(\sqrt{\tau} \cdot u;\tau \right).
\end{equation}
Proposition \ref{fsym} will show that $G$ has the symmetry
\begin{equation} \label{ftran}
  G\left(u ;  1/\tau \right)=\overline{G(\overline{u} ; \overline{\tau})}.
\end{equation}
For each $\tau$, $G(u;\tau)$ is holomorphic in $u$. The notation $G^{(k)}(u;\tau)$  indicates the $k$th derivative with respect to this variable $u$. If $\tau$ is rational then $G(u;\tau)$ has a more explicit description as seen in \e{ggg}.

Employing the methods of Riemann and Siegel, we first extend Theorem \ref{hat0} to the case of general $\alpha$ and $\beta$. This gives the intermediate result, Theorem \ref{hat}, proved in Section \ref{meth}. Interesting  work in a similar  direction to Theorem \ref{hat} is found  in Chapter 4 of  \cite{FJ}, also based on the techniques in \cite{Si32}. Our main theorem, given after the next  definitions, is a completed, symmetric Riemann-Siegel formula that is valid in any vertical strip and that allows the partial sums to have different lengths.
Write the quantity we are interested in estimating as
\begin{equation} \label{rnotn}
  R(s;\alpha,\beta) := e^{i\vartheta(s)}\zeta(s) - e^{i\vartheta(s)}\sum_{n \lqs \alpha} \frac {1}{n^{s}}-e^{i\vartheta(1-s)}\sum_{n\lqs \beta} \frac {1}{n^{1-s}}.
\end{equation}
With \e{thsym} and \e{fesym} it satisfies the symmetry
\begin{equation}\label{refl}
   R(s;\alpha,\beta) = \overline{R(1-\overline{s};\beta,\alpha)}.
\end{equation}
Throughout this work we will use the notation
\begin{equation}\label{notn}
  a:=\alpha-\lfloor \alpha \rfloor, \qquad b:=\beta-\lfloor \beta \rfloor, \qquad \lambda :=\sqrt{\frac \alpha \beta}.
\end{equation}

\begin{theorem} \label{unsym-rs}
Let $I \subset \R$ be a finite interval. Let $s$ be any complex number in the vertical strip described by $\sigma \in I$ and $t\gqs 2\pi$. Then for all $\alpha,\beta \in \R_{\gqs 1}$ with  $t=2\pi \alpha \beta$,  we have
\begin{multline} \label{rrss}
  R(s;\alpha,\beta) =
   (-1)^{\lfloor \alpha \rfloor \lfloor \beta \rfloor+1}
\exp\left(\pi i(2a\beta-2b\alpha+ a^2 \lambda^{-2}  - b^2 \lambda^2)/2\right)
   \\
      \times   \left( \frac{2\pi}{t}\right)^{1/4}   \sum_{n=0}^{N-1} \frac {\lambda^{1/2-s}}{t^{n/2}} \left[ \sum_{r=0}^{3n}
    \frac{G^{(r)}(a  \lambda^{-1} +b \lambda ;\lambda^2 )}{(2\pi)^{r/2}} \cdot P_{n,3n-r}\left(\left(\frac \pi 2\right)^{1/2}(a \lambda^{-1} - b \lambda),\sigma\right)\right]
\\
+O\left( \frac {\lambda^{1/2-\sigma} (\lambda+\lambda^{-1} )^{3N+1/2}}{t^{N/2+1/4}}\right)
\end{multline} 
using the notation from \e{rnotn} and \e{notn}. The implied constant depends only on  $N\in \Z_{\gqs 0}$  and  $I$. The function $G(u;\tau)$ in \e{rrss} is the normalized Mordell integral defined in \e{fdef} and $P_{n,k}(x,\sigma)$ is a polynomial in $x$ and $\sigma$, of degree $k$ in $x$, that is given explicitly in \e{poly}.
\end{theorem}

The simplest case of Theorem \ref{unsym-rs} has $N=0$. Then the sum is empty,  equaling zero, and we find
\begin{equation*}
  R(s;\alpha,\beta) = O\left( \frac {\lambda^{1/2-\sigma} (\lambda+\lambda^{-1} )^{1/2}}{t^{1/4}}\right)
\end{equation*}
 recovering  Hardy and Littlewood's Theorem \ref{hl}. When $N=1$ we obtain the next term in the asymptotic expansion:
\begin{multline} \label{rrss20}
  R(s;\alpha,\beta) =
   (-1)^{\lfloor \alpha \rfloor \lfloor \beta \rfloor+1}
\exp\left(\pi i(2a\beta-2b\alpha+ a^2 \lambda^{-2}  - b^2 \lambda^2)/2\right)
\\
  \times  \left( \frac{2\pi}{t}\right)^{1/4} \lambda^{1/2-s} G(a  \lambda^{-1} +b \lambda ;\lambda^2 )
+O\left( \frac {\lambda^{1/2-\sigma} (\lambda+\lambda^{-1} )^{3+1/2}}{t^{1/2+1/4}}\right)
\end{multline}
since, as we will see, $P_{0,0}(x,\sigma) = 1$. For $N=2$, the next term contains derivatives of $G$ times the polynomials $P_{1,0}(x,\sigma)=-1/3$,
\begin{equation} \label{beto}
   P_{1,1}(x,\sigma)= -i x,  \quad P_{1,2}(x,\sigma)=
 x^2 -i\left(\sigma-\textstyle{\frac 12} \right),  \quad P_{1,3}(x,\sigma)=
\textstyle{\frac{i}3} x^3 + \left(\sigma - \textstyle{\frac 12} \right)x.
\end{equation}
We may take $\lambda$  as fixed in these results but this is not necessary;  Theorem \ref{unsym-rs} produces asymptotics whenever $\lambda$ and $1/\lambda$ have order of magnitude less than $t^{1/6}$.

For  $\sigma=1/2$ and $\alpha = \beta$  in Theorem \ref{unsym-rs} (so that $\lambda=1$ and $\alpha = \sqrt{t/2\pi}$), we recover the Riemann-Siegel formula, Theorem \ref{rsfthm}, as the expression
\begin{equation} \label{rrss40}
  R(s;\alpha,\alpha) =
   (-1)^{\lfloor \alpha \rfloor +1}
   \left( \frac{2\pi}{t}\right)^{1/4} \sum_{n=0}^{N-1} \frac {1}{t^{n/2}} \left[ \sum_{r=0}^{3n}
    \frac{G^{(r)}(2a ;1 )}{(2\pi)^{r/2}} \cdot P_{n,3n-r}\left(0,1/2\right)\right]
+O\left( \frac {1}{t^{N/2+1/4}}\right).
\end{equation}
This agrees with the forms of \e{rt} and  \e{riem} since $G(u;1) = \Psi(u)$ by \e{gu1} and, as shown after Lemma \ref{cong}, the numbers $P_{n,3n-r}\left(0,1/2\right)$ are zero unless $r \equiv 3n \bmod 4$. The more general case of $s$ with $\sigma \in I$ and $\alpha = \beta$ has only the difference that $P_{n,3n-r}\left(0,1/2\right)$ in \e{rrss40} is replaced by $P_{n,3n-r}\left(0,\sigma\right)$ and we obtain a simpler form of \e{fab0adr}.

Our normalizations in Theorem \ref{unsym-rs} are guided by the symmetry \e{refl}. If we define a transformation $\mathcal T$ on functions of $s$, $\alpha$ and $\beta$ as
\begin{equation}\label{tra}
  \mathcal T f(s;\alpha,\beta):= \overline{f(1-\overline{s};\beta,\alpha)}
\end{equation}
then $R(s;\alpha,\beta)$ is invariant under $\mathcal T$. All the components on the right side of \e{rrss} are also invariant under $\mathcal T$. For example $\exp\left(\pi i(2a\beta-2b\alpha+ a^2 \lambda^{-2}  - b^2 \lambda^2)/2\right)$
gets mapped to itself since $\mathcal T$ switches $a$ and $b$ and sends $\lambda$ to $1/\lambda$. That $\mathcal T$ sends $G^{(r)}(a  \lambda^{-1} +b \lambda ;\lambda^2 )$ to itself follows from \e{ftran}; see Proposition  \ref{fsym}. We prove that $P_{n,k}(\sqrt{\pi/2}(a \lambda^{-1} - b \lambda),\sigma)$ is invariant under $\mathcal T$ in Theorem \ref{polss}. The invariance of the right side of \e{rrss} under $\mathcal T$ is  required for the final step in the proof of Theorem \ref{unsym-rs} to obtain the correct error term.

The functions $G(u;\tau)$ appearing in \e{rrss} are Mordell integrals and have many fascinating properties and connections, some of which have only been discovered recently \cite{Zw02,CR15,DRZ17}. See Section \ref{dell} for more details.

The polynomials $P_{n,k}(x,\sigma)$ in \e{rrss} seem to be new and we make an initial study of some of their properties in Section \ref{pnry}. Their description in \e{poly} is given  in terms of   Bernoulli, Bell and Hermite polynomials.

\begin{table}[ht]
\centering
\begin{tabular}{ccc}
\hline
$N$ & Theorem \ref{unsym-rs} & \\
\hline
$1$ & $-0.08\textcolor{gray}{810545388} + 0.10\textcolor{gray}{864755195} i$ & \\
$3$ & $-0.087645\textcolor{gray}{36572} + 0.109362\textcolor{gray}{55272} i$ & \\
$5$ & $-0.087645228\textcolor{gray}{33} + 0.10936268\textcolor{gray}{294} i$ & \\
\hline
 & $-0.08764522824 + 0.10936268305 i$ & $R$\\
\hline
\end{tabular}
\caption{The approximations of Theorem \ref{unsym-rs} to $R=R(1/2+600 i;30/\sqrt{\pi},10/\sqrt{\pi})$.} \label{tb31}
\end{table}

Table \ref{tb31} shows an example of how Theorem \ref{unsym-rs} approximates $R(s;\alpha,\beta)$ for $s=1/2+600 i$ and $\alpha/\beta=3$ (so that $\lambda=\sqrt{3}$). The right side of \e{rrss} for different values of $N$ may be compared with the left side which is displayed in the bottom row. Each decimal is correct  to the accuracy shown. Table \ref{tb31b} shows a similar result at $s=-2+600 i$, outside the critical strip. All the calculations in this paper were carried out using Mathematica.
Section \ref{numb} contains further  examples.

\begin{table}[ht]
\centering
\begin{tabular}{ccc}
\hline
$N$ & Theorem \ref{unsym-rs} & \\
\hline
$1$ & $-0.347\textcolor{gray}{8598947} + 0.4\textcolor{gray}{289646591} i$ & \\
$3$ & $-0.347\textcolor{gray}{8754856} + 0.4059\textcolor{gray}{859119} i$ & \\
$5$ & $-0.3479331\textcolor{gray}{346} + 0.40599\textcolor{gray}{29975} i$ & \\
\hline
 & $-0.3479331128 + 0.4059931509 i$ & $R$\\
\hline
\end{tabular}
\caption{The approximations of Theorem \ref{unsym-rs} to $R=R(-2+600 i;30/\sqrt{\pi},10/\sqrt{\pi})$.} \label{tb31b}
\end{table}

\section{The method of Riemann and Siegel} \label{meth}
\subsection{Initial set-up}
The notation $I$ will always denote a finite interval in $\R$. The well-known families of polynomials we require are those of Bernoulli and Hermite, with generating functions
\begin{equation} \label{behe}
  \frac{t e^{xt}}{e^t-1} = \sum_{n=0}^\infty B_n(x) \frac{t^n}{n!}, \qquad e^{2xt-t^2} = \sum_{n=0}^\infty H_n(x) \frac{t^n}{n!},
\end{equation}
respectively. Both $B_n(x)$ and $H_n(x)$ have degree $n$; the  coefficients of $B_n(x)$ are rational and those of $H_n(x)$ are integral.
For $t>0$ we will also need the power series expansion
\begin{align}
  w(z,s) & := \exp\left( (s-1)\log \left(1+\frac z{\sqrt t}\right)  -  i  \sqrt t z + i  \frac{z^2}2 \right)\label{wzsd}\\
  & \phantom{:}=\sum_{k=0}^\infty a_k(s) \cdot z^k  \qquad \quad (|z|<\sqrt t) . \label{wzs}
\end{align}
The coefficients $a_k(s)$ were given recursively by Siegel as we saw in \e{akrec}. We write them in terms of Bell polynomials in Proposition \ref{akbell}. The next result generalizes Theorem \ref{hat0}.

\begin{theorem} \label{hat} Recall the notation \e{notn}. Let $s$ be any complex number in the vertical strip described by $\sigma \in I$ and $t\gqs 2\pi$. Then for all $\alpha,\beta \in \R_{\gqs 1}$ with  $t=2\pi \alpha \beta$ we have
\begin{multline} \label{fab}
  \zeta(s) = \sum_{n \lqs \alpha} \frac 1{n^{s}}+\chi(s) \sum_{n\lqs \beta} \frac 1{n^{1-s}}
  +
  (-1)^{\lfloor \alpha \rfloor \lfloor \beta \rfloor} \frac {(2\pi)^s e^{\pi i s/2}}{\G(s)(e^{2\pi i s}-1)}
 \left(\frac{2\pi}t\right)^{1/4}\lambda^{1/2-s}\\
   \times
       \exp\left( \frac{\pi i}{2} \Bigl[2a\beta-2b\alpha+ a^2 \lambda^{-2}  - b^2 \lambda^2\Bigr]\right)
\exp\left( \left( \frac s2-\frac 14\right) \log \frac t{2\pi} -\frac{i t}2 -\frac{i \pi}8\right)
   \\
  \times \sum_{k=0}^{N-1} a_k(s)
  \sum_{r=0}^k \binom{k}{r} G^{(r)}(a\lambda^{-1}+b \lambda  ; \lambda^2 )
  \frac{e^{\pi i(k-3r)/4}}{2^{k-r}(2\pi)^{r/2}}H_{k-r} \left( \omega_\lambda\right)
  +O\left( \frac{\lambda^{1-\sigma}t^{-\sigma/2}}{t^{N/6}}\right)
\end{multline}
with $\omega_\lambda := e^{-\pi i/4} \sqrt{\frac \pi 2} ( a \lambda^{-1} - b \lambda)$. The implied constant depends only on $I$, $N \in \Z_{\gqs 0}$ and $\lambda$. If $\lambda \gqs 1$ then the implied constant is independent of $\lambda$.
\end{theorem}

The proof is given in this section and follows the main lines of Siegel's work in \cite[pp. 278-285]{Si32}. See also \cite[Chap. 4]{Ti86} and \cite[Chap. 7]{Ed74}.
For any $m\in \Z_{\gqs 1}$ we begin, as in \cite[Eq. (8)]{Si32},  with
\begin{equation}\label{8}
  \zeta(s)= \sum_{n=1}^m \frac 1{n^{s}}+\frac {(2\pi)^s e^{\pi i s/2}}{\G(s)(e^{2\pi i s}-1)}\int_{C} \frac{z^{s-1}e^{-2\pi i mz}}{e^{2\pi i z}-1}\, dz.
\end{equation}
The contour $C$ starts at $-i \infty$ (with $\arg z = -\pi/2$), moves up the imaginary axis, circles close to $0$ and then returns to $-i \infty$ (with $\arg z = 3\pi/2$) as displayed in Figure \ref{bfig}. Formula \e{8} is valid for all $s\in \C$ and shows that $\zeta(s)$ is  holomorphic everywhere, except for a pole at $s=1$, since $1/(\G(s)(e^{2\pi i s}-1))$ has poles  exactly for $s\in \Z_{\gqs 1}$ and $\int_C$ has zeros for $s\in \Z_{\gqs 2}$. The asymptotics of \e{8} as $t\to \infty$ are obtained with the saddle-point method. The idea (see for example \cite[Chap. 4]{Ol}, \cite{Pet97}, \cite{OSper}) is to move the path of integration so that the main contribution to the integral in \e{8} comes from the neighborhood of the saddle-point of the integrand -- where its derivative with respect to $z$ is zero. For simplicity we just find the saddle-point of the numerator $z^{s-1}e^{-2\pi i mz}$. It is
the value
\begin{equation}\label{9}
  \xi:=\frac{s-1}{2\pi i m} = \frac{t}{2\pi m}+\frac{1-\sigma}{2\pi m}i
\end{equation}
and so we need to move $C$ so that it passes close to $\xi$. A short calculation similar to \e{riew}, \e{ww} shows that
\begin{equation} \label{dec50}
  z^{s-1}e^{-2\pi i mz} = \xi^{s-1}e^{-2\pi i m \xi} \exp\left(\frac{2\pi^2 m^2}{s-1}(z-\xi)^2+O\left((z-\xi)^3\right) \right)
\end{equation}
for $z$ close to $\xi$. Then
\begin{equation*}
  \Re\left( \frac{2\pi^2 m^2}{s-1}(z-\xi)^2\right) = \frac{2\pi^2 m^2}{|s-1|}|z-\xi|^2 \cos\left( 2\arg(z-\xi)-\arg(s-1)\right)
\end{equation*}
and so the directions in which \e{dec50} is decreasing the fastest, as $z$ moves away from $\xi$, are when the cosine is $-1$. For $\arg(s-1)$ close to $\pi/2$ this corresponds to $\arg(z-\xi)$ close to $3\pi/4$ and $-\pi/4$.

The poles of the integrand in \e{8} occur at  integers $j$ with residues essentially $j^{s-1}$. This means that moving $C$ to pass through $\xi$ will add a sum of the form $\sum_{j \lqs t/(2\pi m)} j^{s-1}$, giving the desired second part of the approximate functional equation.

As in the statement of the theorem, we choose $\alpha,\beta \in \R_{\gqs 1}$ with  $t=2\pi \alpha \beta$. Let $ m=\lfloor \alpha \rfloor$.
Then
\begin{equation}\label{xxi}
  \xi \approx \frac{t}{2\pi m} \approx \frac{t}{2\pi \alpha} = \beta
\end{equation}
and, following Riemann, we will use $\beta$ as our base point instead of $\xi$.
Similarly to Siegel we introduce the abbreviations
\begin{equation*}
  \varepsilon : = e^{3\pi i/4}, \qquad
    g(z) :=  \frac{z^{s-1} e^{-2\pi i m z}}{e^{2\pi i z}-1}.
\end{equation*}
The contour of integration $C$ in \e{8} is moved to a new contour $C_\beta$  that encloses exactly the integers from $-\lfloor \beta \rfloor$ to $\lfloor \beta \rfloor$ and passes through $\beta$ in the desired direction of steepest descent $\varepsilon$. As shown in Figure \ref{bfig}, $C_\beta$ is made with five lines.
\SpecialCoor
\psset{griddots=5,subgriddiv=0,gridlabels=0pt}
\psset{xunit=0.8cm, yunit=0.8cm, runit=0.8cm}
\psset{linewidth=1pt}
\psset{dotsize=3.5pt 0,dotstyle=*}
\psset{arrowscale=1.5,arrowinset=0.3}
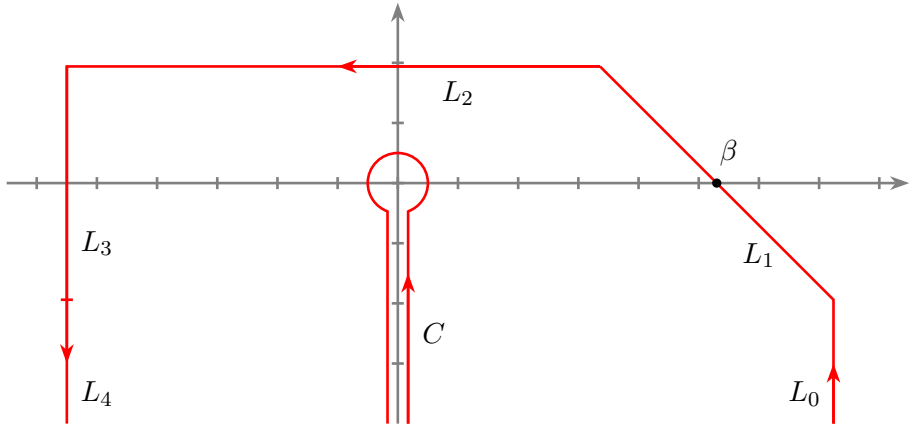
\begin{figure}[ht]
\centering
\begin{pspicture}(-6,-4)(8,3) 

\psline[linecolor=gray]{->}(-6.5,0)(8.5,0)
\psline[linecolor=gray]{->}(0,-4)(0,3)

\psarc[linecolor=red](0,0){0.5}{-70}{250}
\psline[linecolor=red](0.17101,-0.455)(0.17101,-4)
\psline[linecolor=red]{->}(0.17101,-4)(0.17101,-1.5)
\psline[linecolor=red](-0.17101,-0.455)(-0.17101,-4)

\multirput(-6,-0.1)(1,0){15}{\psline[linecolor=gray](0,0)(0,0.2)}
\multirput(-0.1,-3)(0,1){6}{\psline[linecolor=gray](0,0)(0.2,0)}

\psline[linecolor=red]{->}(7.24,-4)(7.24,-3)
\psline[linecolor=red](7.24,-4)(7.24,-1.94)(3.36,1.94)
\psline[linecolor=red]{->}(3.36,1.94)(-1,1.94)
\psline[linecolor=red](3.36,1.94)(-5.5,1.94)(-5.5,-4)
\psline[linecolor=red]{->}(-5.5,1.94)(-5.5,-3)
\psline[linecolor=red](-5.6,-1.94)(-5.4,-1.94)

\psdot(5.3,0)

\rput(0.6,-2.5){$C$}
\rput(-5,-1){$L_3$}
\rput(-5,-3.5){$L_4$}
\rput(1,1.5){$L_2$}
\rput(6,-1.2){$L_1$}
\rput(6.76,-3.5){$L_0$}
\rput(5.5,0.5){$\beta$}


\end{pspicture}
\caption{Contours of integration $C$ and $C_\beta=L_0 \cup L_1 \cup \cdots \cup L_4$}
\label{bfig}
\end{figure}
The first, $L_0$, is the vertical line ending at $\beta-\varepsilon \beta/2$ and then $L_1$ goes from $\beta-\varepsilon \beta/2$ to $\beta+\varepsilon \beta/2$. We require $L_1$ to cross the real line in the open interval $(\lfloor \beta\rfloor,\lfloor \beta\rfloor+1)$; this requires moving the path slightly to the right when $\beta \in \Z$. The horizontal line $L_2$ continues until its real part reaches $-\lfloor \beta \rfloor -1/2$.  The vertical lines $L_3$ and $L_4$ complete the contour with $L_3$ finishing, and $L_4$ starting, level with where $L_0$ finishes.  This is at the imaginary value $-\beta i/(2\sqrt{2})$.

Then
$$
\int_{C} g(z)\, dz = (e^{\pi i s}-1)\sum_{n=1}^{\lfloor \beta \rfloor} \frac 1{n^{1-s}} + \int_{C_\beta} g(z)\, dz
$$
and hence
\begin{equation}\label{13}
  \zeta(s) = \sum_{n\lqs \alpha} \frac 1{n^{s}}+\chi(s) \sum_{n\lqs \beta } \frac 1{n^{1-s}} +\frac {(2\pi)^s e^{\pi i s/2}}{\G(s)(e^{2\pi i s}-1)}\int_{C_\beta} g(z)\, dz .
\end{equation}

\begin{prop} \label{cgz}
For $\alpha,\beta\gqs 1$ and $\sigma \in I$, we have
\begin{equation} \label{20}
  \int_{C_\beta} g(z)\, dz = \int_{L_1} g(z)\, dz +O\left( e^{-t/100}\right).
\end{equation}
The implied constant depends only on $I$.
\end{prop}
\begin{proof}
For the numerator of $g(z)$,
\begin{equation*}
  \left| z^{s-1} e^{-2\pi i m z}\right| = |z|^{\sigma -1}e^{2\pi m y-t \arg z } .
\end{equation*}
Our first claim is that
\begin{equation} \label{clm}
  \left| z^{s-1} e^{-2\pi i m z}\right| \lqs |z|^{\sigma -1} \times \begin{cases}
e^{-t/8} & \text{if $z\in L_2\cup L_3 \cup L_4$},\\
e^{-t/20} & \text{if $z\in L_0$ and $10\lqs \alpha$},\\
e^{\pi |y|} & \text{if $z\in L_0$ and $1\lqs \alpha \lqs 10$}.
\end{cases}
\end{equation}
For $z\in L_2\cup L_3 \cup L_4$,
\begin{equation*}
  y \lqs \frac \beta{2\sqrt 2} \qquad \text{and} \qquad \arg z \gqs \arctan\left( \frac 1{2\sqrt 2-1}\right) >\frac 1{2\sqrt 2} +\frac 18.
\end{equation*}
Hence
\begin{equation*}
  \left| z^{s-1} e^{-2\pi i m z}\right| \lqs |z|^{\sigma -1}\exp\left( \frac{2\pi \alpha \beta}{2\sqrt 2} -\frac t{2\sqrt 2} -\frac t8\right) = |z|^{\sigma -1}e^{-t/8} .
\end{equation*}
With $z\in L_0$  we have
\begin{equation*}
  y \lqs -\frac \beta{2\sqrt 2} \qquad \text{and} \qquad \arg z = -\arctan\left( \frac {|y|}{\beta(1+1/(2\sqrt 2))}\right)
\end{equation*}
so that
\begin{equation}\label{rus}
  \left| z^{s-1} e^{-2\pi i m z}\right|  \lqs |z|^{\sigma -1}\exp\left(
  -2\pi \lfloor \alpha \rfloor |y|+ \frac {t |y|}{\beta(1+1/(2\sqrt 2))}
  \right).
\end{equation}
If $\alpha \gqs 10$ then replacing $\lfloor \alpha \rfloor$ by $\alpha -1$ in \e{rus} shows
\begin{align}
  \left| z^{s-1} e^{-2\pi i m z}\right| & \lqs  |z|^{\sigma -1}\exp\left(
  -t \frac{|y|}\beta\left[\frac{1}{1+2\sqrt 2}-\frac 1\alpha \right]
  \right) \notag\\
  & \lqs |z|^{\sigma -1}\exp\left(
  -\frac t{2\sqrt 2}\left[\frac{1}{1+2\sqrt 2}-\frac 1\alpha \right]
  \right) < |z|^{\sigma -1}e^{-t/20}. \label{245}
\end{align}
If $1 \lqs \alpha \lqs 10$ then writing $2\pi \alpha$ for $t/\beta$ in \e{rus} shows
\begin{equation*}
  \left| z^{s-1} e^{-2\pi i m z}\right|  \lqs  |z|^{\sigma -1}\exp\left(
  2\pi |y|\left[-\lfloor \alpha \rfloor + \frac {\alpha}{1+1/(2\sqrt 2)}
 \right] \right)  < |z|^{\sigma -1}e^{\pi |y|}.
\end{equation*}
This completes the verification of the claim \e{clm}.

For the denominator of $g(z)$:
\begin{align*}
  |e^{2\pi i z}-1|^{-1} \lqs (1-e^{-\pi \beta/\sqrt 2})^{-1} \lqs (1-e^{-\pi/\sqrt 2})^{-1} < 2 & \qquad \text{for} \qquad z \in L_2, \\
  |e^{2\pi i z}-1|^{-1} = (1+e^{-2\pi y})^{-1} < 1 & \qquad \text{for} \qquad z \in L_3.
\end{align*}
Hence, for $z$ in $L_2 \cup L_3$
  we have that $g(z) \ll \beta^{\sigma -1}e^{-t/8}$.
Therefore
\begin{equation*}
  \int_{L_2 \cup L_3} g(z)\, dz  \ll \beta^{\sigma}e^{-t/8} \ll t^\sigma e^{-t/8} \ll e^{-t/20}.
\end{equation*}

For $z$ in $L_0 \cup L_4$ we have
\begin{equation*}
  |z| < 4|y| \qquad \text{and} \qquad  |e^{2\pi i z}-1|^{-1} < 2e^{-2\pi|y|}.
\end{equation*}
Therefore
\begin{equation*}
  \int_{L_4} g(z)\, dz   \ll e^{-t/20} \int_{1/(2\sqrt{2})}^\infty e^{-2\pi y} y^{\sigma-1} \,dy \ll e^{-t/20},
\end{equation*}
and we obtain the same bound for $\int_{L_0} g(z)\, dz$ when $10\lqs \alpha$. In the final case with $1 \lqs \alpha \lqs 10$,
\begin{equation*}
  \int_{L_0} g(z)\, dz  \ll  \int_{\beta/(2\sqrt{2})}^\infty e^{-\pi y} y^{\sigma-1} \,dy \ll \int_{t/(40 \sqrt{2}\pi)}^\infty e^{-2\pi y/3} \,dy \ll e^{-t/100}. \qedhere
\end{equation*}
\end{proof}

\subsection{The saddle-point method}
 The work in this paper grew out of the project \cite{OSper} which aimed to clarify some aspects of the saddle-point method, as elegantly formulated by Perron in 1917. The paper \cite{Pet97} documents that   this method originated with Riemann, and it is remarkable that one of his first applications was to finding the asymptotic expansion for the difficult case of $\zeta(s)$.

In simpler applications of the saddle-point method, such as \cite[Cor. 1.4]{OSper}, the part of the integrand containing the growing parameter $N$ is expanded into the form $\exp(N c (z-\xi)^2)$ times a power series in $z$ about the fixed saddle-point $\xi$. The behaviour of the integral for $z$ close to $\xi$ will control the asymptotics. Our case is more difficult as the saddle-point  \e{9} is not fixed and changes with the  parameters $s$ and $\alpha$. Adding to the complications, it is inconvenient to expand  the numerator of $g(z)$ about $\xi$ and we expand about the nearby point $\beta$ instead:
\begin{equation} \label{riew}
  z^{s-1} e^{-2\pi i m z} = \beta^{s-1}e^{-2\pi i m\beta}
   \exp\left(-2\pi i m(z-\beta)+ (s-1)\log \left(1+\frac{z-\beta}\beta\right)\right).
\end{equation}
The argument of $\exp$ above may be developed as
\begin{equation} \label{ww}
  -2\pi i m(z-\beta) -(s-1)\sum_{j=1}^\infty \frac{(-1)^j}{j}\left(\frac{z-\beta}{\beta}\right)^j.
\end{equation}
When $z$ is close to $\beta$ we find \e{ww} is
\begin{multline}\label{ww2}
  \left(\frac {\sigma-1+i t}\beta -2\pi i m\right)(z-\beta) + \frac{1-\sigma-i t}{2}\left(\frac{z-\beta}{\beta}\right)^2 + O\left(\left(\frac{z-\beta}{\beta}\right)^3\right) \\
  =\left[i\left(\frac t\beta -2\pi m\right)(z-\beta) -  \frac{i t}{2\beta^2}(z-\beta)^2\right]
  + (\sigma-1) \left(\frac{z-\beta}{\beta}\right) -\frac{\sigma-1}{2} \left(\frac{z-\beta}{\beta}\right)^2 + O\left( \left(\frac{z-\beta}{\beta}\right)^3\right).
\end{multline}
For $\beta$ large, the piece
\begin{equation*}
  \left[i\left(\frac t\beta -2\pi m\right)(z-\beta) -  \frac{i t}{2\beta^2}(z-\beta)^2\right]
  = 2\pi i(\alpha -m)(z-\beta) -\pi i \frac \alpha \beta (z-\beta)^2
\end{equation*}
of \e{ww2} will be biggest. Therefore we separate it out and, recalling \e{wzsd}, write
\begin{equation*}
  z^{s-1} e^{-2\pi i m z} = \beta^{s-1}e^{-2\pi i m\beta} \exp\left(   -\pi i \frac \alpha \beta (z-\beta)^2+ 2\pi i(\alpha -m)(z-\beta)\right) w\left(\frac{\sqrt t}{\beta}(z-\beta),s\right).
\end{equation*}
Since $\alpha/\beta=\lambda^2$ and $\sqrt t/\beta=\sqrt{2\pi}\lambda$, we obtain
\begin{equation} \label{il1}
  \int_{L_1} g(z)\, dz = \beta^{s-1}e^{-2\pi i m\beta} \int_{L_1} \frac{\exp\left(   -\pi i \lambda^2 (z-\beta)^2+ 2\pi i a(z-\beta)\right)}{e^{2\pi i z}-1}w\left(\sqrt{2\pi}\lambda(z-\beta),s\right)\, dz  .
\end{equation}

The next step is to replace $w(z,s)$ by the first terms in its expansion \e{wzs}.
The tail of this power series has the
presentation
\begin{equation}\label{20b}
  r_n(z,s):=\sum_{k=n}^\infty a_k(s) z^k =\frac{z^n}{2\pi i}\int_{\mathcal C} \frac{ w(u,s)}{u^n (u-z)} \, du,
\end{equation}
with $\mathcal C$  a curve inside the disc $|u|<\sqrt{t}$ which encircles $0$ and $z$ in the positive direction. Siegel bounded $r_n(z,s)$ precisely and for completeness we include his proof \cite[pp. 281-282]{Si32} since all the error bounds depend on it.

\begin{lemma} \label{rnzs}
For $n \in \Z_{\gqs 0}$, $\sigma \in I$ and $t>0$ we have the estimates
\begin{alignat}{2} \label{24}
  r_n(z,s) & = O\left( \frac{|z|^n}{t^{n/6}}\right) &\qquad \text{for} &\qquad 1\lqs n \lqs \frac{27}{50}t, \ |z|\lqs \frac{20}{21}\left(\frac{2n  \sqrt{t}}{5} \right)^{1/3} , \\
  r_n(z,s) & = O\left( e^{14|z|^2/29}\right) &\qquad  \text{for} & \qquad |z| \lqs  \sqrt{t}/2 . \label{25}
\end{alignat}
The implied constants depend only on $I$ and $n$.
\end{lemma}
\begin{proof}
With \e{wzsd},
\begin{equation*}
  \log w(u,s) = (\sigma -1) \log \left(1+\frac u{\sqrt{t}}  \right) +i u^2 \sum_{k=1}^\infty \frac{(-1)^{k-1}}{k+2} \left( \frac u{\sqrt{t}} \right)^k.
\end{equation*}
Therefore, in the circle $| u | \lqs 3\sqrt{t} /5$ we have
\begin{equation}\label{21}
  \Re \left(\log w(u,s)\right) \lqs |\sigma -1| \log \frac 85 +\frac{5|u|^3}{6\sqrt{t} } .
\end{equation}
In \e{20b} let $|z|\lqs 4\sqrt{t} /7$ and let $\mathcal C$ be a circle around $u = 0$ with a radius $\rho_n$ satisfying
\begin{equation}\label{22}
  \frac{21}{20}|z| \lqs \rho_n \lqs \frac 35 \sqrt{t} .
\end{equation}
Then \e{20b}, \e{21}, \e{22} imply
the estimate
\begin{equation}\label{23}
  r_n(z,s) = O\left( |z|^n \rho_n^{-n} \exp\left( \frac 5{6\sqrt{t} } \rho_n^3\right)\right).
\end{equation}
For $n\gqs 1$, the function $\rho^{-n} e^{5\rho^3/(6 \sqrt{t}) }$ of $\rho$ reaches its minimum $\left(\frac{5e}{2n \sqrt{t} } \right)^{n/3}$ for $\left(\frac{2n \sqrt{t} }{5} \right)^{1/3}$.
According to \e{22}, the choice $\rho_n=\rho$ is admissible if
\begin{equation*}
\frac{21}{20}|z| \lqs \left(\frac{2n \sqrt{t} }{5} \right)^{1/3}\lqs \frac 35 \sqrt{t} .
\end{equation*}
Consequently we obtain \e{24}. For $n\gqs 0$ and $|z|\lqs 4\sqrt{t} /7$, the choice $\rho_n =21|z|/20$ is also admissible according to \e{22}; from this we obtain \e{25}.
\end{proof}

\subsection{Error estimates}
If we replace $w(\cdot,s)$ in \e{il1} by the first $n$ terms of its expansion \e{wzs} then the error involves the integral
\begin{equation} \label{jn}
  J_n(s;\alpha,\beta):=\int_{L_1} \frac{\exp\left(   -\pi i \lambda^2 (z-\beta)^2+ 2\pi i a(z-\beta)\right)}{e^{2\pi i z}-1}r_n\left(\sqrt{2\pi} \lambda(z-\beta),s\right)\, dz  .
\end{equation}

\begin{prop} \label{longj}
Suppose $\sigma \in I$, $t>0$ and $n \in \Z_{\gqs 0}$. Then we have
\begin{equation} \label{jt}
  J_n(s;\alpha,\beta) = O\left( t^{-n/6}\right)
\end{equation}
for an implied constant depending only on $I$, $n$ and $\lambda$. If $\lambda\gqs 1$ then the implied constant does not depend on $\lambda$.
\end{prop}
\begin{proof}
Recall that $L_1$ is usually a straight line from $\beta-\varepsilon \beta/2$ to $\beta+\varepsilon \beta/2$. However, when $\beta$ is  close to $\lfloor \beta \rfloor$ or $\lfloor \beta \rfloor +1$ we will adjust the path slightly to avoid the denominator in \e{jn} becoming too small. The next lemma has a straightforward proof that is omitted.

\begin{lemma} \label{lemb}
Suppose $\delta>0$ and $z\in \C$. If $|z-m| \gqs \delta$ for all $m\in \Z$ then,
for an absolute implied constant,
\begin{equation*}
  (e^{2\pi i z}-1)^{-1} = O\left(1+\delta^{-1}\right).
\end{equation*}
\end{lemma}

The proof of the proposition breaks into
four cases.

\vskip 3mm
\noindent
{\bf Case I:} $\lambda \lqs 1$ and $1/100 \lqs b \lqs 99/100$. For these values of $b$ we may take $L_1$ to be a straight line. The part of the integrand $(e^{2\pi i z}-1)^{-1}$ in \e{jn} is absolutely bounded as we may apply Lemma \ref{lemb} with $\delta=1/(100\sqrt{2})$.

Writing $z=\beta+\varepsilon v/(\sqrt{2\pi} \lambda)$ shows that
\begin{equation} \label{fgh}
  J_n(s;\alpha,\beta) = \frac{\varepsilon}{\sqrt{2\pi} \lambda}\int_{-\sqrt{t}/2}^{\sqrt{t}/2} \frac{\exp\left(-v^2/2 +\sqrt{2\pi}  i \varepsilon a v/\lambda \right)}{e^{2\pi i(\beta+\varepsilon v/(\sqrt{2\pi} \lambda))}-1} r_n\left( \varepsilon v,s \right)\, dv
\end{equation}
and the integrand is bounded by a constant times
\begin{equation} \label{alex}
  \exp\left(-v^2/2 +\sqrt{2\pi} |v|/\lambda \right) \left| r_n\left( \varepsilon v,s \right)\right|.
\end{equation}
Using \e{25} we have
\begin{equation} \label{tro}
  J_n(s;\alpha,\beta)  \ll  \int_{0}^{\sqrt{t}/2} \exp\left(-\frac{v^2}{58}+\frac{\sqrt{2\pi}}{\lambda} v\right)\,dv  =O(1)
\end{equation}
 with an implied constant depending on $n$ and  $\lambda$. If $n=0$ then \e{tro} gives the correct bound \e{jt}.

 Now we fix $n \in \Z_{\gqs 1}$.
Assume $t\gqs 50n/27$ so that we may also use \e{24}.
Let $\mu := \frac{20}{21}(2n\sqrt{t}/5)^{1/3}$ and applying the bounds \e{24} and \e{25} to \e{alex} shows
\begin{align}
   J_n(s;\alpha,\beta) & \ll t^{-n/6} \int_0^\mu \exp\left(-\frac{v^2}{2}+\frac{\sqrt{2\pi}}{\lambda} v\right) v^n\, dv+ \int_{\mu}^{\sqrt{t}/2}\exp\left(-\frac{v^2}{58}+\frac{\sqrt{2\pi}}{\lambda} v\right)\,dv \notag\\
& \ll t^{-n/6} \int_0^\infty \exp\left(-\frac{v^2}{3}\right) v^n\, dv+ \int_{\mu}^{\infty}\exp\left(-\frac{v^2}{59}\right)\,dv \notag\\
& \ll t^{-n/6} + e^{-\mu^2/60} =O(t^{-n/6}) \qquad \text{for} \qquad t\gqs 50n/27. \label{tro2}
\end{align}
By \e{tro} we have
\begin{equation}\label{tro3}
  J_n(s;\alpha,\beta)=O(1) \qquad \text{for} \qquad 0<t< 50n/27.
\end{equation}
 Combining \e{tro2} and \e{tro3} gives the desired bound \e{jt}.

\vskip 3mm
\noindent
{\bf Case II:} $\lambda \lqs 1$ and $0\lqs b \lqs 1/100$ or $99/100 \lqs b <1$. For these values of $b$ we
let $L_1$ be the usual path of integration except that we replace the segment between $\beta-\varepsilon/50$ and  $\beta+\varepsilon/50$ with a semicircular arc of radius $1/50$ about $\beta$. If $0\lqs b \lqs 1/100$ we need the upper arc traversed in a counter-clockwise direction. For $99/100 \lqs b <1$ we need the lower arc traversed in a clockwise direction. We focus on the former case  from here and it is shown in Figure \ref{cfig}. The other case is similar.
\SpecialCoor
\psset{griddots=5,subgriddiv=0,gridlabels=0pt}
\psset{xunit=0.6cm, yunit=0.6cm, runit=0.6cm}
\psset{linewidth=1pt}
\psset{dotsize=3.5pt 0,dotstyle=*}
\psset{arrowscale=1.5,arrowinset=0.3}
\begin{figure}[ht]
\centering
\begin{pspicture}(-5,-3.5)(5,3.5) 

\psline[linecolor=gray]{->}(-5,0)(5.5,0)
\psline[linecolor=gray](-0.5,0.3)(-0.5,-0.3)

\pscircle[linecolor=black,linestyle=dotted,dotsep=1pt](-0.5,0){1.5}

\psarc[linecolor=red](0,0){3}{-45}{135}
\psline[linecolor=red](3.5,-3.5)(2.1213,-2.1213)
\psline[linecolor=red]{->}(3.5,-3.5)(2.7,-2.7)
\psline[linecolor=red](-3.5,3.5)(-2.1213,2.1213)
\psline[linecolor=red]{->}(-2.1213,2.1213)(-3,3)

\rput(3.5,-1){$L_1$}
\rput(-2.5,-1.9){radius $1/100$}

\psdot(0,0)

\rput(0.2,0.5){$\beta$}
\rput(-0.5,-0.7){$\lfloor \beta \rfloor$}


\end{pspicture}
\caption{Adjusting the contour of integration $L_1$ near $\lfloor \beta \rfloor$ in Case II}
\label{cfig}
\end{figure}
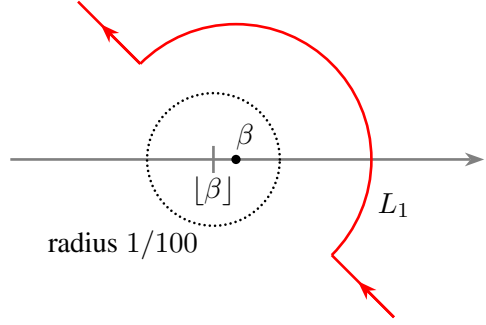

Since the circle of radius $1/100$ about $\lfloor \beta \rfloor$ is contained in the circle of radius $1/50$ about $\beta$  we see that $z\in L_1$ satisfies Lemma \ref{lemb} with $\delta =1/100$. Therefore $(e^{2\pi i z}-1)^{-1}$  is absolutely bounded and the work of Case I shows the correct bound for the part of $J_n$ given by the integral on the straight lines.

Let $J_n(s;\alpha,\beta)_{\mathcal A}$ be the remaining part of $J_n$ given by the integral over the arc:
\begin{equation} \label{jnw}
  J_n(s;\alpha,\beta)_{\mathcal A} =\int_{\mathcal A} \frac{\exp\left(   -\pi i \lambda^2 w^2+ 2\pi i a w\right)}{e^{2\pi i (\beta+w)}-1}r_n\left(\sqrt{2\pi} \lambda w,s\right)\, dw
\end{equation}
where $\mathcal A$ is given by $w$ with $|w|=1/50$ and $-\pi/4 \lqs \arg w \lqs 3\pi/4$.
The integrand is bounded by a constant times
\begin{equation*}
  \exp\left(   \pi  \lambda^2 /50^2+ 2\pi /50\right) \left|r_n\left(\sqrt{2\pi} \lambda w,s\right)\right|
\end{equation*}
Using \e{25} we have
\begin{equation} \label{hro}
  J_n(s;\alpha,\beta)_{\mathcal A}  \ll  \int_{\mathcal A} \exp\left( 14\left| \sqrt{2\pi}\lambda w\right|^2 /29\right) \, |dw|  =O(1)
\end{equation}
 with an implied constant depending on $n$ and  $\lambda$. If $n=0$ then \e{hro} gives the correct bound for \e{jt}.

Now we fix $n \in \Z_{\gqs 1}$.
Assume $t\gqs 50n/27$ so that, by \e{24},
\begin{equation*} 
  J_n(s;\alpha,\beta)_{\mathcal A}  \ll  \int_{\mathcal A} \left|r_n\left(\sqrt{2\pi} \lambda w,s\right)\right| \, |dw| \ll \left(\frac{\sqrt{2\pi} \lambda}{50}\right)^n t^{-n/6}  =O\left(t^{-n/6}\right)
\end{equation*}
provided that $\sqrt{2\pi} \lambda/50 \lqs \mu$. (Recall $\mu$ defined before \e{tro2}.) But this last condition is true if $t>1.4\times 10^{-7}$. Hence
\begin{equation}\label{hro3}
  J_n(s;\alpha,\beta)_{\mathcal A} =O\left(t^{-n/6}\right)  \qquad \text{for} \qquad t\gqs 50n/27
\end{equation}
and by \e{hro}
\begin{equation}\label{hro4}
  J_n(s;\alpha,\beta)_{\mathcal A} =O\left(1\right)  \qquad \text{for} \qquad 0<t< 50n/27.
\end{equation}
The estimates \e{hro3} and \e{hro4} complete the proof of \e{jt} in this case.

\vskip 3mm
\noindent
{\bf Case III:} $\lambda \gqs 1$ and $1/(100\lambda) \lqs b \lqs 99/(100\lambda)$. This is similar to Case I, though we are more careful in showing the $\lambda$ dependence. The integration path $L_1$ is straight with no adjustments. For $z\in L_1$, the part of the integrand $(e^{2\pi i z}-1)^{-1}$ in \e{jn} is  bounded by an absolute constant times $1+\lambda$ as we may apply Lemma \ref{lemb} with $\delta=1/(100\sqrt{2}\lambda)$.

As in \e{fgh} and \e{alex}, it may be seen that $J_n(s;\alpha,\beta)$ is bounded by an absolute constant times
\begin{equation} \label{frt}
   \frac{1+\lambda}{\lambda}\int_{0}^{\sqrt{t}/2} \exp\left(-v^2/2 +\sqrt{2\pi} |v|/\lambda \right) \left| r_n\left( \varepsilon v,s \right)\right|\, dv.
\end{equation}
As \e{frt} is decreasing in $\lambda$, we may take $\lambda=1$ in our bounds and the error will not depend on $\lambda$. The arguments of Case I now go through unchanged.

\vskip 3mm
\noindent
{\bf Case IV:} $\lambda \gqs 1$ and $0\lqs b \lqs 1/(100\lambda)$ or $99/(100\lambda) \lqs b <1$. Similarly to Case II, $L_1$ is the usual path of integration except that we replace the segment between $\beta-\varepsilon/(50\lambda)$ and  $\beta+\varepsilon/(50\lambda)$ with a semicircular arc of radius $1/(50\lambda)$ about $\beta$. As in Case II, we may focus on the situation with $0\lqs b \lqs 1/(100\lambda)$.

Since the circle of radius $1/(100\lambda)$ about $\lfloor \beta \rfloor$ is contained in the circle of radius $1/(50\lambda)$ about $\beta$  we see that $z\in L_1$ satisfies Lemma \ref{lemb} with $\delta =1/(100\lambda)$. Therefore $(e^{2\pi i z}-1)^{-1}$  is  bounded by an absolute constant times $1+\lambda$ and the work of Case III shows the correct bound for the part of $J_n$ given by the integral on the straight lines. Let $J_n(s;\alpha,\beta)_{\mathcal A}$ be the remaining part of $J_n$ given by the integral over the arc:
\begin{equation} \label{jnw2}
  J_n(s;\alpha,\beta)_{\mathcal A} =\frac 1\lambda\int_{\mathcal A} \frac{\exp\left(   -\pi i  w^2+ 2\pi i a w/\lambda\right)}{e^{2\pi i (\beta+w/\lambda)}-1}r_n\left(\sqrt{2\pi} w,s\right)\, dw
\end{equation}
where $\mathcal A$ was already defined for \e{jnw} and given by $w$ with $|w|=1/50$ and $-\pi/4 \lqs \arg w \lqs 3\pi/4$. Hence, $J_n(s;\alpha,\beta)_{\mathcal A}$ is bounded by an absolute constant times
\begin{equation} \label{jnw3}
 \frac{1+\lambda}\lambda \int_{\mathcal A} \exp\left(   \frac{\pi}{50^2}+ \frac{2\pi}{50\lambda}\right) \left|r_n\left(\sqrt{2\pi} w,s\right)\right|\, dw.
\end{equation}
Then \e{jnw3} is decreasing in $\lambda$ and so we may reuse the estimates of Case II with $\lambda=1$ to bound $J_n(s;\alpha,\beta)_{\mathcal A}$ as in \e{hro3} and \e{hro4}.
\end{proof}

With \e{il1}, \e{jn} and Proposition \ref{longj}, we have shown that
\begin{multline} \label{inn}
  \int_{L_1} g(z)\, dz
  = \beta^{s-1}e^{-2\pi i m\beta} \sum_{k=0}^{N-1} a_k(s)(2\pi)^{k/2}\lambda^k \\
\times
   \int_{L_1}  \frac{\exp\left(   -\pi i \lambda^2 (z-\beta)^2+ 2\pi i a(z-\beta)\right)}{e^{2\pi i z}-1} (z-\beta)^k \, dz
 +O\left( \beta^{\sigma-1}t^{-N/6}\right)
\end{multline}
for $\sigma \in I$, $t>0$ and $\lambda>0$, with an implied constant depending only on $I$, $N\in \Z_{\gqs 0}$ and $\lambda$ when $\lambda<1$.

The last step in rearranging our expressions for $\int_{L_1} g(z)\, dz$ is to extend the line of integration on the right of \e{inn} to infinity in both directions. Let $L^-$ be the line from $\beta-\varepsilon \infty$ to $\beta-\varepsilon \beta/2$ and let $L^+$ be the line from $\beta+\varepsilon \beta/2$ to $\beta+\varepsilon \infty$.

\begin{lemma} \label{inft}
For $\sigma \in I$, $t\gqs 1$ and $N\in \Z_{\gqs 0}$ we have
\begin{equation} \label{jlem}
  \sum_{k=0}^{N-1} a_k(s) (2\pi)^{k/2}\lambda^k
   \int_{L^- \cup L^+}  \frac{\exp\left(   -\pi i \lambda^2 (z-\beta)^2+ 2\pi i a(z-\beta)\right)}{e^{2\pi i z}-1} (z-\beta)^k \, dz  = O\left( e^{-t/16}\right).
\end{equation}
The implied  constant depends only on $I$, $N$ and $\lambda$. If $\lambda \gqs 1$ then the implied  constant is independent of $\lambda$.
\end{lemma}
\begin{proof}
We first note that for $z \in L^- \cup L^+$ it is true that
\begin{equation*}
  |\Im z| \gqs \frac{\beta}{2\sqrt 2} = \frac{\sqrt t}{2\sqrt 2 \cdot \sqrt{2\pi} \lambda} \gqs \frac{1}{4\sqrt{\pi}\lambda}.
\end{equation*}
Hence by Lemma \ref{lemb}, $(e^{2\pi i z}-1)^{-1} = O\left(1+\lambda\right)$ for an absolute implied constant. Next we see by using \e{25} with $z=1/2$ that
\begin{equation} \label{aksbnd}
  a_k(s) = \left(r_k(z,s) -  r_{k+1}(z,s)\right) z^{-k} =O(1)
\end{equation}
(where this implied constant depends on $k$ and $I$).

With the change of variables $z=\beta+\varepsilon v/(\sqrt{2\pi} \lambda)$ that we used in \e{fgh}, the left side of \e{jlem} equals
\begin{equation*}
  \sum_{k=0}^{N-1} a_k(s) \frac{\varepsilon}{\sqrt{2\pi}\lambda}
   \int_{(-\infty,-\sqrt{t}/2]\cup [\sqrt{t}/2,\infty)}  \frac{\exp\left(-v^2/2 +\sqrt{2\pi}  i \varepsilon a v/\lambda \right)}{e^{2\pi i(\beta+\varepsilon v/(\sqrt{2\pi} \lambda))}-1} (\varepsilon v)^k \, dv
\end{equation*}
and this is bounded by an absolute constant times
\begin{equation} \label{port}
  \sum_{k=0}^{N-1} |a_k(s)| \frac{1+\lambda}{\lambda}
   \int_{\sqrt{t}/2}^{\infty} \exp\left(-v^2/2 +\sqrt{2\pi}   v/\lambda \right) v^k \, dv.
\end{equation}
Then \e{port}  is
\begin{equation*}
  \ll \int_{\sqrt{t}/2}^{\infty} \exp\left(-v^2/3 \right) \, dv \ll \exp\left(-(\sqrt{t}/2)^2/4 \right)
\end{equation*}
as required. The implied constant in \e{jlem} is independent of $\lambda$ for $\lambda\gqs 1$ because \e{port} is decreasing in $\lambda$.
\end{proof}

\subsection{Relating the integral to $G(u;\tau)$}
\begin{proof}[Proof of Theorem \ref{hat}]
It follows from Lemma \ref{inft} that \e{inn} is true with the path of integration $L_1$ extended to infinity.
If we replace $z$ by $z+\lfloor \beta \rfloor$ in the integral in \e{inn} then it is easy to see that the path of integration may now be taken as any infinite straight line  crossing the real line in the interval $(0,1)$ and in the direction of $\varepsilon$. As before, we use $0\nwarrow 1$ to denote this path.

Then combining this with \e{13} and Proposition \ref{cgz} gives
\begin{multline} \label{130}
  \zeta(s) = \sum_{n\lqs \alpha} n^{-s}+\chi(s) \sum_{n\lqs \beta} n^{s-1}+\frac {(2\pi)^s e^{\pi i s/2}}{\G(s)(e^{2\pi i s}-1)} \beta^{s-1}e^{-2\pi i m\beta} \sum_{k=0}^{N-1} a_k(s)(2\pi)^{k/2}\lambda^k \\
\times
   \int_{0 \nwarrow 1}  \frac{\exp\left(   -\pi i \lambda^2 (z-b)^2+ 2\pi i a(z-b)\right)}{e^{2\pi i z}-1} (z-b)^k \, dz
 +O\left( \frac{\beta^{\sigma-1}t^{1/2-\sigma}}{t^{N/6}}\right).
\end{multline}
For the error, we used that
\begin{equation*}
  \frac {(2\pi)^s e^{\pi i s/2}}{\G(s)(e^{2\pi i s}-1)} =O\left( t^{1/2-\sigma}\right)
\end{equation*}
for $\sigma \in I$ and $t\gqs 2\pi$ by Stirling's formula or  Proposition \ref{stir2}. Writing the error in \e{130} in terms of $\lambda$ gives the error stated in \e{fab}.

The integral in \e{130} may be expressed in terms of $G(u;\tau)$. It is simpler to relate the integral directly to the Mordell integral $\Upsilon(u;\tau)$ in \e{mord}, but this would obscure the symmetry in $s \leftrightarrow 1-s$, $\alpha \leftrightarrow \beta$ we wish to exploit.
The integral in \e{130}  is $k!(2\pi i)^{-k}$ times the coefficient of $u^k$ in
\begin{multline}\label{pos}
   \int_{0 \nwarrow 1}  \frac{\exp\left(   -\pi i \lambda^2 (z-b)^2+ 2\pi i a(z-b)\right)}{e^{2\pi i z}-1}\exp\bigl(2\pi i(z-b)u\bigr) \, dz \\
  = \lambda^{-1/2}\exp\left(-\pi i \left(b^2 \lambda^2  + \frac 18 \right)\right)
\exp\left( \frac{\pi i }{2}\left(\frac{a}{\lambda}-b \lambda +\frac{u}{\lambda}\right)^2\right)
  G\left(\frac{a}{\lambda}+ b \lambda  +\frac{u}{\lambda}; \lambda^2\right).
\end{multline}
It follows from the right identity in \e{behe} that
\begin{equation*}
  e^{-c^2(u+q)^2} = e^{-c^2 q^2}\sum_{n=0}^\infty H_n(cq) \frac{(-c)^n u^n}{n!}.
\end{equation*}
Expanding the exponential factor on the right of \e{pos} using this,
with $c=e^{-\pi i/4} \sqrt{ \pi /2}$ and $q=a/\lambda-b \lambda$, produces
\begin{equation*}
\exp\left( \frac{\pi i }{2}\left(\frac{a}{\lambda}-b \lambda +\frac{u}{\lambda}\right)^2\right)
 = \exp\left( \frac{\pi i}{2}(a \lambda^{-1}-b \lambda)^2\right) \sum_{n=0}^\infty H_n \left( \omega_\lambda \right) \left(\frac{\pi  }{2} \right)^{n/2}\left(\frac{\varepsilon}{\lambda} \right)^{n}\frac{ u^n}{n!}.
\end{equation*}
Combining this with the Taylor expansion of $G\left(a/\lambda+ b \lambda  +u/\lambda; \lambda^2\right)$ shows the integral in \e{130} equals
\begin{multline} \label{for}
  \frac 1{(2\pi i)^k} \lambda^{-1/2-k}
\exp\left(-\pi i \left(b^2 \lambda^2  + \frac 18 \right)\right)
 \exp\left( \frac{\pi i}{2}(a \lambda^{-1}-b \lambda)^2\right)\\
\times\sum_{r=0}^k \binom{k}{r} G^{(r)}\left(a \lambda^{-1}+b \lambda; \lambda^2\right)
  H_{k-r} \left( \omega_\lambda  \right)\left(\frac{\pi  }{2} \right)^{(k-r)/2}\varepsilon^{k-r}.
\end{multline}
Put \e{for}  into \e{130} and the final step, to get the formula into the form we want, uses the identities
\begin{equation}\label{bilb}
  \frac{\beta^{s-1}}{\lambda^{1/2}} = \left(\frac{2\pi}t\right)^{1/4} \exp\left( \left( \frac s2-\frac 14\right) \log \frac t{2\pi}\right)\lambda^{1/2-s}
\end{equation}
and
\begin{multline} \label{for2}
  \exp\left(-2\pi i m \beta -\pi i \left(b^2 \lambda^2  + \frac 18 \right)\right)
 \exp\left( \frac{\pi i}{2}(a \lambda^{-1}-b \lambda)^2\right)
 \\
 =(-1)^{\lfloor \alpha \rfloor \lfloor \beta \rfloor}
       \exp\left( \frac{\pi i}{2} \Bigl[2a\beta-2b\alpha+ a^2 \lambda^{-2}  - b^2 \lambda^2\Bigr]\right)
\exp\left(  -\frac{i t}2 -\frac{i \pi}8\right).
\end{multline}
Equation \e{bilb} follows from $\beta^2=t/(2\pi \lambda^2)$. Also $m = \lfloor \alpha \rfloor$,
 $\pi i\alpha \beta = i t/2$
and the equalities
\begin{align*}
  1  =e^{2\pi i \lfloor \alpha \rfloor \lfloor \beta \rfloor} & = e^{2\pi i(\alpha -a)(\beta-b)} =  e^{2\pi i(\alpha \beta -\alpha b -a\beta+ab)},\\
(-1)^{\lfloor \alpha \rfloor \lfloor \beta \rfloor} & = e^{\pi i(\alpha - a)(\beta - b)} = e^{\pi i(\alpha \beta -\alpha b -a\beta+ab)}
\end{align*}
show \e{for2}. This completes the proof  of Theorem \ref{hat}.
\end{proof}

\section{The Mordell integral $\Upsilon(u;\tau)$} \label{dell}

\subsection{Relations} \label{aup}
For $u,\tau \in \C$ with $\Re(\tau)>0$, recall our definition from \e{mord}
\begin{equation} \label{morddd}
  \Upsilon(u;\tau):=\int_{0 \nwarrow 1} \frac{e^{-\pi i \tau z^2 +2\pi i u z}}{e^{2\pi i z}-1} \, dz.
\end{equation}
This  type of integral was studied in detail by  Mordell \cite{Mo33} with earlier work by Kronecker, Lerch, Ramanujan and Riemann; see the references in the introduction of \cite{CR15}. The method of Riemann for $\tau=1$, described in \cite[Sect. 1]{Si32}, easily extends to give
\begin{align}
  \Upsilon(u+1;\tau) & = \Upsilon(u;\tau) +\tau^{-1/2} e^{\pi i(u^2/\tau+3/4)}, \label{up1}\\
  \Upsilon(u+\tau;\tau) & = e^{\pi i(\tau+2u)}\left(\Upsilon(u;\tau) -1 \right).\label{upt}
\end{align}
Applying these relations repeatedly yields the following result.
\begin{prop}
Suppose $u,\tau \in \C$ with $\Re(\tau)>0$. Then for all $m,n \in \Z_{\gqs 0}$
\begin{align}\label{psi}
  \Upsilon(u+m;\tau) &  = \Upsilon(u;\tau) +e^{3\pi i/4} \tau^{-1/2}  \sum_{j=0}^{m-1} e^{\pi i (j+u)^2/\tau},  \\
  \Upsilon(u+n\tau;\tau) & = e^{\pi i n(n\tau+2u)}\Upsilon(u;\tau) - e^{\pi i (n \tau+u)^2/\tau}
  \sum_{j=0}^{n-1} e^{-\pi i (j \tau+u)^2/\tau}. \label{psi2}
\end{align}
\end{prop}
This allows us to compute $\Upsilon(u;\tau)$ explicitly for $\tau$ a positive rational. If $n\tau=m$ for $m,n \in \Z_{\gqs 1}$ then equating \e{psi} and \e{psi2} shows, as in \cite[Sect. 1]{De67},
\begin{equation} \label{rat}
  \left( e^{\pi i n(m+2u)}-1\right) \Upsilon\left(u;\frac mn\right) = e^{3\pi i/4} \frac{\sqrt{n}}{\sqrt{m}}\sum_{j=0}^{m-1} e^{\pi i (j+u)^2n/m}+ e^{\pi i (m +u)^2 n/m}
  \sum_{j=0}^{n-1} e^{-\pi i (j +n u/m)^2 m/n}.
\end{equation}
The right side of \e{rat} is left essentially unchanged if $m$ and $n$ are interchanged, $u$ is replaced by $n u/m$, and everything is conjugated. Precisely, we have
\begin{equation*}
  \overline{\left( e^{\pi i m(n+2n \overline{u}/m)}-1\right) \Upsilon\left(\frac{n\overline{u}}m;\frac nm\right)} =  e^{-3\pi i/4} \frac{\sqrt{m}}{\sqrt{n}} e^{-\pi i (m +u)^2 n/m} \left( e^{\pi i n(m+2u)}-1\right) \Upsilon\left(u;\frac mn\right).
\end{equation*}
Simplifying this shows
\begin{align}
  \overline{ \Upsilon\left(\frac{n\overline{u}}m;\frac nm\right)} & = -e^{-3\pi i/4} \frac{\sqrt{m}}{\sqrt{n}} e^{-\pi i (m +u)^2 n/m} e^{\pi i n(m+2u)} \Upsilon\left(u;\frac mn\right) \notag\\
  & = \frac{\sqrt{m}}{\sqrt{n}} e^{-\pi i (nu^2/m-1/4)} \Upsilon\left(u;\frac mn\right). \label{silon}
\end{align}
As in \e{fdef},  set
\begin{equation} \label{gdef}
  G(u;\tau) := \tau^{1/4}\exp\left( -\frac{\pi i u^2}{2}+\frac{\pi i}{8}\right)  \Upsilon \left(\sqrt{\tau} \cdot u;\tau \right).
\end{equation}
This definition gives the simplest possible transformation under $\tau \to 1/\tau$, as we see next.

\begin{prop} \label{fsym}
For all $u,\tau \in \C$ with $\Re(\tau)>0$ we have
\begin{align} \label{uwp}
  \Upsilon( u/\tau; 1/\tau  ) & = \sqrt{\tau} e^{\pi i(u^2/\tau-1/4)} \overline{ \Upsilon (\overline{u};\overline{\tau} )}, \\
  G(u ;1/\tau  ) & =\overline{G(\overline{u} ; \overline{\tau})}, \label{uwp2}\\
G^{(k)}(u ; 1/\tau ) & = \overline{G^{(k)}(\overline{u} ; \overline{\tau})}, \qquad (k \in \Z_{\gqs 0}). \label{fft}
\end{align}
\end{prop}
\begin{proof}
We obtain \e{uwp} from \e{silon} for all $u\in \C$ and all $\tau \in \Q_{>0}$. Since both sides of \e{uwp} are holomorphic functions of $\tau$ for  $\Re(\tau)>0$, it follows that \e{uwp} extends to all these values of $\tau$. Then \e{uwp2} follows directly from \e{uwp}. Differentiating \e{uwp2} with respect to $u$ provides \e{fft}.
\end{proof}

\subsection{Examples}
Let
\begin{equation} \label{thek}
  \theta_{k}(u):=u^2/2 -\sqrt{k} u-k/2-1/8.
\end{equation}
For all $m,n \in \Z_{\gqs 1}$ we have from \e{rat} and \e{gdef} that
\begin{multline} \label{ggg}
  G\left(u;\frac mn\right) = \frac{1}{2i \sin(\pi (\sqrt{mn} u +mn/2))} \\
\times\left[
\left(\frac mn\right)^{1/4}  \exp\Bigl(-\pi i \theta_{mn}(u)\Bigr) \sum_{j=0}^{n-1} \exp\left(-\pi i j\left[2u\sqrt{\frac mn}+j \frac mn\right]\right)
\right.\\
\left.- \left(\frac mn\right)^{-1/4}  \exp\Bigl(\pi i \theta_{mn}(u)\Bigr) \sum_{j=0}^{m-1} \exp\left(\pi i j\left[2u\sqrt{\frac nm}+j \frac nm\right]\right) \right].
\end{multline}
If $u$ makes $\sqrt{mn} u +mn/2$ an integer, then the denominator $\sin(\pi (\sqrt{mn} u +mn/2))$ is zero. Since $G(u;m/n)$ is a holomorphic function of $u$, it follows that the numerator in \e{ggg} must also be zero. For these values of $u$, $G(u;m/n)$ may be found by taking limits. The numerator being zero in these cases also gives instances of Gauss sum reciprocity as mentioned by Siegel in \cite{Si32} and shown in \cite{De67}.

In the simplest case of $\tau=1$  we know by  \e{uwp2} that $G(u;1)=\overline{G(\overline{u};1)}$ and so $G(u;1)$ is real-valued when $u\in \R$. Then \e{ggg} implies
\begin{equation*}
  G(u;1)=-\frac{1}{2i \sin(\pi (u+1/2))}\left(e^{\pi i \theta_1(u)}- e^{-\pi i \theta_1(u)}\right) = -\frac{\sin(\pi \theta_1(u) )}{\sin(\pi (u+1/2))}.
\end{equation*}
This may also be written as
\begin{equation} \label{gu1}
  G(u;1)= -\frac{\sin(\pi (u^2/2-u-5/8) )}{\sin(\pi (u+1/2))} = \frac{\cos(\pi (u^2/2-u-1/8 ))}{\cos(\pi u)}.
\end{equation}
For $\tau=2,$ $3$ we find
\begin{align}
  G(u;2) & =\frac{-1}{2i \sin(\sqrt{2}\pi u)}\left[2^{1/4}e^{-\pi i \theta_2(u)} - 2^{-1/4}e^{\pi i \theta_2(u)}\left(1+i e^{\sqrt{2}\pi i u}\right) \right], \label{gu2}\\
G(u;3) & =\frac{-1}{2i \cos(\sqrt{3}\pi u)}\left[3^{1/4}e^{-\pi i \theta_3(u)} - 3^{-1/4}e^{\pi i \theta_3(u)}\left(1+ e^{\pi i(2 u/\sqrt{3}+1/3)} + e^{\pi i(4 u/\sqrt{3}+4/3)}\right) \right]. \notag
\end{align}

\subsection{Analytic continuation and  Zwegers' $h(u;\tau)$}  \label{bup}

The results in this subsection will not be needed in the rest of the paper, though they establish an interesting connection.
In his thesis \cite{Zw02}, \cite[Chap. 8]{BFOR}, Zwegers puts the mock theta functions of Ramanujan into a modular framework. His Appell-Lerch series $\mu(z_1,z_2;\tau)$, see \cite[Sect. 1.3]{Zw02}, is a two-variable Jacobi form of weight $1/2$, except that a term containing
\begin{equation} \label{zweh}
  h(u;\tau):= \int_{-\infty}^\infty \frac{e^{\pi i \tau y^2 +2\pi u y}}{\cosh(\pi y)} \, dy \qquad \quad (u\in \C, \ \Im(\tau)>0)
\end{equation}
appears in its modular transformations.
Zwegers then shows that, by adding a non-holomorphic component, $\mu$ may be completed into a Jacobi form that transforms correctly. The mock theta functions can then be expressed in terms of $\mu$, as in \cite[Appendix A]{BFOR}.

The function $h(u;\tau)$ also  appears in \cite[Sect. 1.5]{Zw02} as the period integral of a weight $3/2$ unary theta function. Ramanujan wrote $h(u;\tau)$ in terms of a partial theta function; see \cite[Thm. 2.1]{CR15}. Thus, $h(u;\tau)$ has many interesting connections to modular and mock modular forms. We see next that $h(u;\tau)$ and $\Upsilon(u;\tau)$ are closely related.

\begin{lemma}
We have
\begin{equation} \label{uups}
  \Upsilon(u;\tau)=e^{\pi i(u-1/2-\tau/4)}\int_{-\infty}^\infty \frac{e^{\pi i \tau y^2 +\pi  y(\tau-2u+1)}}{e^{\pi y}+e^{-\pi y}} \, dy
\end{equation}
for all $u$, $\tau$ with
\begin{equation} \label{kav}
  \Re(\tau)>0, \quad \Im(\tau)>0 \quad \text{and} \quad \frac{\Re(\tau)-\Im(\tau)}2 < \Re(u) < 1+\frac{\Re(\tau)-\Im(\tau)}2.
\end{equation}
\end{lemma}
\begin{proof}
We wish to rotate the line of integration $0 \nwarrow 1$ in \e{mord} and make it vertical, passing through $1/2$. For large $Y>0$ we replace the line of integration from $1/2$ to $1/2-Y+iY$ by the lines from $1/2$ to $1/2+iY$ and $1/2+iY$ to $1/2-Y+iY$. To bound the integral on the horizontal segment
\begin{equation*}
  \mathfrak I_Y := \int_{1/2+iY}^{1/2-Y+iY} \frac{e^{-\pi i \tau z^2 +2\pi i u z}}{e^{2\pi i z}-1} \, dz,
\end{equation*}
we let $z=x+iY$ and find
\begin{equation*}
  \mathfrak I_Y \ll \int_{1/2-Y}^{1/2} \frac{\exp\left(2\pi \left[ \Re(\tau) x Y + \Im(\tau)(x^2-Y^2)/2-\Re(u) Y -\Im(u)x \right] \right)}{e^{-2\pi Y}-1} \, dx.
\end{equation*}
For $\Im(\tau)>0$ we have $\Im(\tau)(x^2-Y^2) \lqs \Im(\tau)(1/4-Y)$, and so obtain
\begin{align}
  \mathfrak I_Y  & \ll \exp\Bigl(-\pi Y(2\Re(u)+\Im(\tau))\Bigr) \int_{1/2-Y}^{1/2} \exp\left(2\pi x\left[ \Re(\tau)  Y -\Im(u) \right] \right) \, dx \notag\\
   & \ll \exp\Bigl(-\pi Y(2\Re(u)+\Im(\tau)-\Re(\tau))\Bigr). \label{kav2}
\end{align}

The line of integration from  $1/2+Y-iY$  to $1/2$ is also replaced by horizontal and vertical lines. A similar argument shows that the horizontal integral satisfies
\begin{equation} \label{kav3}
  \mathfrak I_{-Y} := \int^{1/2-iY}_{1/2+Y-iY} \frac{e^{-\pi i \tau z^2 +2\pi i u z}}{e^{2\pi i z}-1} \, dz
\ll \exp\Bigl(-\pi Y(-2\Re(u)+2-\Im(\tau)+\Re(\tau))\Bigr).
\end{equation}
Therefore, as $Y \to \infty$, the bounds \e{kav2}, \e{kav3} imply that  $\mathfrak I_Y \to 0$ and  $\mathfrak I_{-Y} \to 0$ if the inequalities on the right of \e{kav} are satisfied. This completes the proof.
\end{proof}

Hence
\begin{equation} \label{uph}
  \Upsilon(u;\tau)= \frac 12 e^{\pi i(u-1/2-\tau/4)} h\left(\frac{\tau}{2}-u+\frac 12;\tau \right),
\end{equation}
initially for all $u$ and $\tau$ satisfying \e{kav}. By analytically continuing both sides in $u$ we see that \e{uph} becomes true for all $u\in \C$ when $\Re(\tau),$ $\Im(\tau)>0$. Therefore  \e{uups} and \e{uph} give the analytic continuation of $\Upsilon(u;\tau)$ to all $\tau$ with $\Im(\tau)>0$ (and \e{uph} gives the analytic continuation of $h(u;\tau)$ to all $\tau$ with $\Re(\tau)>0$).

 Conjugating both sides of  \e{uups}  shows that, for all $u\in \C$ and all $\tau$ with $\Im(\overline{\tau})>0$, we have
\begin{equation}\label{conju}
 \overline{\Upsilon(\overline{u};\overline{\tau})} = e^{\pi i(\tau-2u+1)}\Upsilon(u-\tau; -\tau) .
\end{equation}
Rearranging and  simplifying with \e{upt} shows, for $\Im(\tau)>0$,
\begin{equation}\label{conju2}
 \Upsilon(u;\tau) = 1- \overline{\Upsilon(\overline{u};-\overline{\tau})}.
\end{equation}
Since $\Upsilon(u;\tau)$ exists for $\Re(\tau)>0$, the relation \e{conju2} provides the continuation of   $\Upsilon(u;\tau)$ to all $\tau$ with $\Re(\tau)<0$. In this way we have extended the definition of  $\Upsilon(u;\tau)$ to all $u\in \C$ and all $\tau \in \C$ except for $\tau$ on the negative imaginary axis: $(-\infty,0]i$.
It follows that \e{conju2} is valid for all $\tau$ outside $(-\infty,0]i$. Numerically, the values of $\Upsilon(u;\tau)$ for $\tau$  on each side of $(-\infty,0]i$  do not match, so we may take it as a branch cut.
We have shown:

\begin{prop}
For each $u\in \C$, the function $\Upsilon(u;\tau)$ defined in \e{morddd} is an analytic function of $\tau$ when $\Re(\tau)>0$. With \e{zweh} and \e{uph} we obtain the analytic continuation to $\Im(\tau)>0$. Then \e{conju2} gives the continuation to  $\Re(\tau)<0$.
\end{prop}

Combining \e{conju} with \e{uwp} shows
\begin{equation*}
  e^{\pi i(\tau-2u+1)}\Upsilon(u-\tau;-\tau) = \frac 1{\sqrt{\tau}} e^{-\pi i(u^2/\tau-1/4)} \Upsilon( u/\tau; 1/\tau  ) \qquad (\Re(\tau)>0).
\end{equation*}
This is also
\begin{equation} \label{sil}
  \Upsilon( -u/\tau; -1/\tau  ) = \sqrt{-\tau} e^{-\pi i(u^2/\tau+2u+\tau-3/4)} \Upsilon(u+\tau;\tau) \qquad (\Re(\tau)<0).
\end{equation}
For $\Re(\tau)<0$ and $\Im(\tau)>0$ we have the equality of principal square roots $\sqrt{-\tau}=e^{-\pi i/4}\sqrt{-i\tau}$. Putting this into \e{sil} gives
\begin{equation} \label{sil2}
  \Upsilon( -u/\tau; -1/\tau  ) = \sqrt{-i\tau} e^{-\pi i(u^2/\tau+2u+\tau-1/2)} \Upsilon(u+\tau;\tau)
\end{equation}
which by analytic continuation in $\tau$ is valid for all $\tau$ outside of the negative imaginary axis.
Translating \e{sil2} into a relation for $h(z;\tau)$ by \e{uph}, with $z=u+\tau/2-1/2$ and using $h(-z;\tau)=h(z;\tau)$, shows
\begin{equation}\label{hhh}
  h(z/\tau;-1/\tau)= \sqrt{-i\tau} e^{-\pi i z^2/\tau}h(z;\tau)
\end{equation}
which is part $(5)$ of \cite[Prop. 1.2]{Zw02}. Therefore  we have given another proof of \e{hhh} which is proved in  \cite{Zw02} with the Fourier transform. Alternatively, starting with \e{hhh} and using \e{conju}, we may give another proof of Proposition \ref{fsym}.

Parts $(1)$ and $(2)$ of \cite[Prop. 1.2]{Zw02} are equivalent to \e{up1}, \e{upt}.
Part $(6)$ translates into the following interesting identity. For all $u\in \C$ and initially for all $\tau$ with $\Im(\tau)>0$
\begin{equation*}
  \Upsilon(u;\tau)-\Upsilon\left(u+\frac 12;\tau+1\right) = \frac{1}{\sqrt{\tau+1}} \exp\left( \pi i \left[\frac{4u^2-4u-\tau}{4(\tau+1)} \right]\right) \cdot \Upsilon\left(\frac {2u+\tau}{2(\tau+1)};\frac{\tau}{\tau+1}\right).
\end{equation*}

\subsection{Bounds for $\Upsilon$ and $G$}
The proof of our main theorem will require these next estimates.
\begin{prop} \label{upb}
For all $u,\tau \in \R$ with $0<\tau \lqs 1$ we have
\begin{equation*}
  \tau^{k/2} \Upsilon^{(k)}(u;\tau)\ll \tau^{-1/2}(1+|u|)\left(1+ \frac{1+|u|^k}{\tau^{k/2}} \right)
\end{equation*}
for an implied constant depending only on $k\in \Z_{\gqs 0}$.
\end{prop}
\begin{proof}
Suppose that $u=x+m$ with $0\lqs x <1$ and $m\in \Z$. If $m\gqs 0$, then differentiating \e{psi} $k$ times implies
\begin{equation*}
  \Upsilon^{(k)}(u;\tau)  = \Upsilon^{(k)}(x;\tau) +e^{3\pi i/4} \tau^{-1/2}  \sum_{j=0}^{m-1} \frac{d^k}{dx^k} e^{\pi i (j+x)^2/\tau}.
\end{equation*}
The right-hand terms may be evaluated with the identity
\begin{equation}\label{herm}
  \frac{d^k}{dx^k} e^{-c^2(x+q)^2} = \left[(-c)^k H_k(c(x+q)) \right] e^{-c^2(x+q)^2}
\end{equation}
and $c=e^{-\pi i/4}\sqrt{\pi/\tau}$, $q=j$. Hence
\begin{align}
   \Upsilon^{(k)}(u;\tau)  - \Upsilon^{(k)}(x;\tau) & = \tau^{-1/2}  \sum_{j=0}^{m-1} e^{3\pi i(k+1)/4}\left(\frac \pi \tau\right)^{k/2}  H_k\left(e^{-\pi i/4}\frac{\sqrt{\pi}}{\sqrt{\tau}}(x+j)\right) e^{\pi i (j+x)^2/\tau} \label{brex}\\
   & \ll \tau^{-k/2-1/2}  \sum_{j=0}^{m-1}\left(1+\frac{(j+1)^k}{\tau^{k/2}} \right)\notag\\
   & \ll \tau^{-k/2-1/2} (1+|u|)\left( 1+\frac{(1+|u|)^k}{\tau^{k/2}}\right). \label{morb}
\end{align}
We find the same bound when $u=x+m$ with $m\lqs 0$. This reduces the question to estimating $\Upsilon^{(k)}(u;\tau)$ for $0\lqs u<1$.

Differentiating \e{mord} inside the integral is valid and writing $z=1/2+\varepsilon t$ then shows
\begin{equation} \label{updiff}
  \Upsilon^{(k)}(u;\tau) = -\varepsilon (\pi i)^ke^{\pi i(u-\tau/4)} \int_\R \frac{\exp\left( -\pi \tau t^2+\pi i \varepsilon(2u-\tau) t\right)}{e^{2\pi i \varepsilon t}+1} (1+2\varepsilon t)^k\, dt.
\end{equation}
It is straightforward to see that
\begin{equation} \label{un}
  \frac{1}{|e^{2\pi i \varepsilon t}+1|} < \begin{cases}
  2e^{\sqrt{2}\pi t} & \text{if} \quad t\lqs 0,\\
  2 & \text{if} \quad t\gqs 0.
  \end{cases}
\end{equation}
We have
\begin{equation} \label{ann}
  \Upsilon^{(k)}(u;\tau)\ll  \int_\R \exp\left( -\pi \tau t^2 +\frac{\pi \tau t}{\sqrt 2} \right)\frac{\exp\left(-\sqrt 2 \pi u t\right)}{|e^{2\pi i \varepsilon t}+1|} (1+|t|^k) \, dt.
\end{equation}
If we now assume that $0\lqs u \lqs 1$, then the middle fraction in \e{ann} is at most $2$ by \e{un}. Changing variables we obtain
\begin{equation*}
  \Upsilon^{(k)}(u;\tau)\ll  \int_\R \exp\left( -\pi v^2 +\frac{\pi \sqrt\tau v}{\sqrt 2} \right) \left(1+\frac{|v|^k}{\tau^{k/2}}\right) \, \frac{dv}{\tau^{1/2}}
\end{equation*}
and this is $\ll \tau^{-1/2}(1+\tau^{-k/2})$ when $\tau\lqs 1$.
Using this last bound for $\Upsilon^{(k)}(x;\tau)$ in \e{morb} and simplifying completes the proof.
\end{proof}

\begin{theorem}
For all $u,\tau \in \R$ with $\tau>0$ we have
\begin{equation} \label{mar}
  G^{(k)}(u;\tau) \ll \begin{cases}
\tau^{-1/4}\left( 1+\tau^{1/2}|u|\right)\left(1+\tau^{-k/2}+ |u|^{k}\right)\phantom{\Big|} & \text{if} \quad \tau \lqs 1,\\
\tau^{1/4}\left( 1+\tau^{-1/2}|u|\right)\left(1+\tau^{k/2}+ |u|^{k}\right)\phantom{\Big|}  & \text{if} \quad \tau \gqs 1
\end{cases}
\end{equation}
for an implied constant depending only on $k\in \Z_{\gqs 0}$.
\end{theorem}
\begin{proof}
From the definition \e{fdef} and \e{herm} we have
\begin{multline}
  G^{(k)}(u;\tau)  :=\tau^{1/4}e^{\pi i/8}\sum_{j=0}^k \binom{k}{j} \frac{d^j}{du^j}\exp\left( -\frac{\pi i u^2}{2}\right) \cdot \frac{d^{k-j}}{du^{k-j}}\Upsilon(\sqrt{\tau}u;\tau) \\
 = \tau^{1/4}\exp\left( -\frac{\pi i u^2}{2}+\frac{\pi i}8\right) \sum_{j=0}^k \binom{k}{j} e^{5\pi i j/4} \left(\frac{\pi}{2}\right)^{j/2} H_j\left(e^{\pi i/4}\frac{\sqrt{\pi}}{\sqrt{2}} u \right) \cdot \tau^{(k-j)/2}\Upsilon^{(k-j)}(\sqrt{\tau}u;\tau). \label{fwy}
\end{multline}
Then, using  Proposition \ref{upb},
\begin{align}
  G^{(k)}(u;\tau) & \ll \tau^{-1/4}\sum_{j=0}^k \left( 1+|u|^j\right) \left(1+\tau^{1/2}|u|\right)\left(1+ \tau^{(j-k)/2}+|u|^{k-j}\right)\notag\\
& \ll \tau^{-1/4}\left( 1+\tau^{1/2}|u|\right)\left(1+\tau^{-k/2}+ |u|^{k}\right) \label{xhu}
\end{align}
for $\tau \lqs 1$. 
When $\tau \gqs 1$, the relation \e{fft} combined with \e{xhu} finishes the proof of \e{mar}.
\end{proof}

\begin{cor} \label{lamb}
For $\lambda>0$ and $a,b$ satisfying $0\lqs a,b\lqs 1$ we have
\begin{equation*}
  G^{(k)}(a \lambda^{-1}+b\lambda;\lambda^2) =O\left( \lambda^{k+1/2}+\lambda^{-k-1/2}\right)
\end{equation*}
for an implied constant depending only on $k\in \Z_{\gqs 0}$.
\end{cor}

\subsection{Linear independence}
The next result will be needed in the proof of Theorem \ref{polss}.

\begin{prop} \label{ffff}
Let $\tau$ with $\Re(\tau)>0$ be fixed. The functions of $u$ in the set
\begin{equation*}
  \left\{G(u;\tau), \ G^{(1)}(u;\tau), \ G^{(2)}(u;\tau), \ \dots , \ G^{(m)}(u;\tau)\right\}
\end{equation*}
are linearly independent for any $m\in \Z_{\gqs 0}$. (The $m=0$ case is saying that, for each $\tau$, $G(u;\tau)$ is never identically $0$ as a function of $u$.)
\end{prop}
\begin{proof}
Suppose we have
\begin{equation} \label{cfs}
  \sum_{j=0}^m c_j \tau^{-j/2}G^{(j)}(u;\tau)=0 \qquad \quad (u\in \C).
\end{equation}
The constants $c_j$ may depend on the fixed $\tau$ and  it is convenient to include the nonzero factor $\tau^{-j/2}$.
Replacing $G(u;\tau)$ with $\Upsilon(\sqrt{\tau} u;\tau)$ using \e{fwy}, and then replacing $u$ by $u/\sqrt{\tau}$ implies
\begin{equation} \label{sumup}
   \sum_{j=0}^m \psi_j(u) \Upsilon^{(j)}(u;\tau)=0 \qquad \quad (u\in \C)
\end{equation}
for polynomials $\psi_j(u)$. Explicitly, for $0\lqs j\lqs m$, we have
\begin{equation*}
  \psi_j(u) =  \sum_{k=j}^m c_k  \binom{k}{j} e^{5\pi i (k-j)/4}\left(\frac{\pi}{2\tau}\right)^{(k-j)/2} H_{k-j}\left(e^{\pi i/4}\frac{\sqrt{\pi}}{\sqrt{2\tau}} u \right).
\end{equation*}
The highest degree term in $H_n(y)$ is $2^n y^n$ by \e{hann}
and so $\psi_j(u)$ has degree $m-j$ with highest degree term $c_m \binom{m}{j} (-\pi i /\tau)^{m-j} \cdot u^{m-j}$.

Since we know \e{up1}, it is natural to apply the difference operator $\Delta$ to \e{sumup}. We have
\begin{equation*}
  \Delta \Upsilon^{(j)}(u;\tau):= \Upsilon^{(j)}(u+1;\tau)-\Upsilon^{(j)}(u;\tau) =L_j(u) e^{\pi i u^2/\tau}
\end{equation*}
for, using the calculation in \e{brex} with $m=1$,
\begin{equation*}
  L_j(u):=\tau^{-1/2} e^{3\pi i(j+1)/4}\left(\frac \pi \tau\right)^{j/2} H_j\left(e^{-\pi i/4} \frac{\sqrt{\pi}}{\sqrt{\tau}} u\right) .
\end{equation*}
Recall that $\Delta (f(u)g(u))=(\Delta f(u))\cdot g(u) +f(u+1)\cdot (\Delta g(u))$. Hence $\Delta$  applied to \e{sumup} implies
\begin{equation} \label{ms}
  \sum_{j=0}^{m-1}\left(\Delta \psi_j(u)\right) \Upsilon^{(j)}(u;\tau) +\sum_{j=0}^m \psi_j(u+1)L_j(u) e^{\pi i u^2/\tau} =0
\end{equation}
and a second application gives
\begin{equation*}
  \sum_{j=0}^{m-2}\left(\Delta^2 \psi_j(u)\right) \Upsilon^{(j)}(u;\tau)
+\sum_{j=0}^{m-1} \left(\Delta \psi_j(u+1)\right) L_j(u) e^{\pi i u^2/\tau}
+\sum_{j=0}^m \Delta\left[\psi_j(u+1)L_j(u) e^{\pi i u^2/\tau}\right] =0.
\end{equation*}
Clearly, applying $\Delta$ to a polynomial reduces the degree by at least $1$.
After a total of $m+1$ applications of the difference operator to  \e{sumup}, the functions $\Upsilon^{(j)}(u;\tau)$ disappear and we are left with
\begin{equation}\label{1am}
  \sum_{k=0}^m\sum_{j=0}^{m-k} \Delta^{m-k}\left[\left(\Delta^k \psi_j(u+1)\right)L_j(u) e^{\pi i u^2/\tau}\right] =0.
\end{equation}
Dividing both sides of \e{1am} by $e^{\pi i u^2/\tau}$ implies
\begin{equation} \label{zero}
  \sum_{k=0}^m \phi_k(u) e^{2\pi i k u/\tau} = 0
\end{equation}
for polynomials $\phi_k(u)$. Clearly $e^{2\pi i m u/\tau}$ only appears in \e{1am} when $k=0$. Noting that
\begin{equation*}
  \Delta^m(f(u)g(u))=\sum_{j=0}^m \binom{m}{j} \Delta^{m-j}(f(u+j))\cdot \Delta^j g(u)
\end{equation*}
we find
\begin{equation*}
  \phi_m(u) = e^{\pi i m^2/\tau}\sum_{j=0}^{m} \psi_j(u+m+1)L_j(u+m).
\end{equation*}

The degree of $\phi_m(u)$ is $m$ since $\psi_j(u)$ has degree $m-j$ and $L_j(u)$ has degree $j$. The coefficient of $u^m$ in $\phi_m(u)$ is therefore
\begin{multline} \label{zero2}
  e^{\pi i m^2/\tau}\sum_{j=0}^{m} \left[c_m \binom{m}{j} \left(\frac{-\pi i }{\tau}\right)^{m-j} \right]\cdot \left[\tau^{-1/2}  e^{3\pi i/4}\left(\frac {2\pi i} \tau\right)^{j}\right]  \\
  = c_m \cdot  \tau^{-1/2}  e^{\pi i (m^2/\tau+3/4)} \sum_{j=0}^{m} \binom{m}{j}  \left(\frac{-\pi i }{\tau}\right)^{m-j} \left(\frac {2\pi i} \tau\right)^{j}\\
= c_m \cdot \tau^{-1/2} e^{\pi i (m^2/\tau+3/4)} \left(\frac{\pi i }{\tau}\right)^{m}.
\end{multline}
However, it follows from \e{zero} that all the polynomials $\phi_k(u)$ are identically $0$. A simple way to see this is to put $u=-i\tau y$ and examine the size of each term as $y\to \infty$. Hence \e{zero2} is $0$ and so $c_m=0$. Repeating this argument shows that all of the coefficients $c_0,c_1, \dots,c_m$ in \e{cfs} are $0$, as we wanted to prove.
\end{proof}

\section{Some series expansions}

\subsection{The Riemann-Siegel  function $\vartheta(s)$} \label{rsf}
We begin with 
\begin{lemma} \label{bin}
Suppose $s\in \C$ satisfies $\sigma \in I$ and $t \neq 0$. Then for all $k,R \in \Z_{\gqs 0}$ we have
\begin{equation} \label{tay}
  \frac{1}{(\sigma+it)^k} = \frac 1{(it)^k}\sum_{r=0}^{R-1}\binom{k+r-1}{r} \left( \frac{-\sigma}{it}\right)^r +O\left(\frac{1}{|t|^{k+R}}\right)
\end{equation}
for an implied constant depending only on $I$, $k$ and $R$.
\end{lemma}
\begin{proof}
By Taylor's Theorem and bounding the integral form of the remainder in the usual way, as in \e{20b}, we have
\begin{equation} \label{summ}
  (1+z)^{-k} = \sum_{r=0}^{R-1}\binom{-k}{r}z^r +O(|z|^R)
\end{equation}
for all $z\in \C$ when $|z|\lqs 1/2$, say. The coefficients of $z^r$ in the sum \e{summ} are given by the Generalized Binomial Theorem. With $z=\sigma/(it)$, this proves \e{tay} when $|t|\gqs 2|\sigma|$. If $0< |t| \lqs 2|\sigma|$ then
\begin{multline*}
 |t|^{k+R} \left| \frac{1}{(\sigma+it)^k} - \frac 1{(it)^k}\sum_{r=0}^{R-1}\binom{k+r-1}{r} \left( \frac{-\sigma}{it}\right)^r\right|\\
\lqs |t|^R+\sum_{r=0}^{R-1}\binom{k+r-1}{r} |\sigma|^r |t|^{R-r}
\lqs |2\sigma|^R+\sum_{r=0}^{R-1}\binom{k+r-1}{r} |\sigma|^r |2\sigma|^{R-r} = O(1).
\end{multline*}
Hence the error in \e{tay} is  $O(|t|^{-k-R})$ for $0< |t| \lqs 2|\sigma|$ as well, completing the proof.
\end{proof}

Hermite and Barnes gave the asymptotics of $\log \G(z+a)$ as $|z| \to \infty$ when $0\lqs a \lqs 1$. These shifted argument results and further improvements are described in \cite{Ne}. See also \cite[Ex. 4.4, p. 295]{Ol}, for example.
Our next proposition shows the asymptotics of $\log \G(s)$, for $s$  in any vertical strip, in terms of $\sigma$ and $t$. It agrees with the previously mentioned work when $0\lqs \sigma \lqs 1$.

\begin{prop} \label{logg}
Suppose $s\in \C$ satisfies $\sigma \in I$ and $t \neq 0$. Then
\begin{equation*}
  \log \G(s) = \left(s-\frac 12\right)\log it - it +\frac 12 \log 2\pi
   -\sum_{k=1}^{N-1} \left( \frac it \right)^k \frac{ B_{k+1}(\sigma)}{k(k+1)}
+O\left( \frac{1}{|t|^{N}}\right)
\end{equation*}
for an implied constant depending only on $I$ and $N\in \Z_{\gqs 1}$.
\end{prop}
\begin{proof}
Stirling's series as in \cite[p. 294]{Ol} states that
for all $s\in \C$ with  $s\notin (-\infty,0]$ we have
\begin{equation} \label{stir}
  \log \G(s)-\left(s-\frac 12\right)\log s + s -\frac {\log 2\pi}2 = \sum_{n=1}^{M-1} \frac{B_{2n}}{2n(2n-1)}
\frac 1{s^{2n-1}} -\frac{1}{2M}\int_0^\infty \frac{B_{2M}(v-\lfloor v\rfloor)}{(v+s)^{2M}}\, dv
\end{equation}
where  $M\in \Z_{\gqs 1}$.
We may replace the last term in \e{stir} with $O(1/|t|^{2M-1})$ since
\begin{equation*}
  \int_0^\infty \frac{|B_{2M}(v-\lfloor v\rfloor)|}{|v+s|^{2M}}\, dv \ll \int_{-\infty}^\infty \frac{1}{((v+\sigma)^2+t^2)^{M}}\, dv = \sqrt{\pi}\frac{\G(M-1/2)}{\G(M)}\frac 1{|t|^{2M-1}}
\end{equation*}
where the last equality is \cite[3.241.4]{GR}.
With Lemma \ref{bin} we may write each $1/s^{2n-1}$ term in \e{stir} as
\begin{equation*}
  \frac 1{s^{2n-1}} = \sum_{r=0}^{R_n-1} \binom{2n+r-2}{r}\frac{(-i)^{2n-1}(i\sigma)^r}{t^{2n+r-1}}+O\left( \frac{1}{|t|^{2n+R_n-1}}\right).
\end{equation*}
Choosing $R_n$ so that $2n+R_n-1=2M-1$ we find that the left side of \e{stir} equals
\begin{multline*}
  \sum_{n=1}^{M-1} \frac{B_{2n}}{2n(2n-1)}
\sum_{r=0}^{2M-2n-1} \binom{2n+r-2}{r}\frac{(-i)^{2n-1}(i\sigma)^r}{t^{2n+r-1}}
+O\left( \frac{1}{|t|^{2M-1}}\right)\\
  =-\sum_{k=1}^{2M-2} \frac{i^k}{t^k}\sum_{n=1}^{\lfloor (k+1)/2\rfloor} \binom{k-1}{2n-2} \frac{B_{2n}}{2n(2n-1)}\sigma^{k+1-2n}+O\left( \frac{1}{|t|^{2M-1}}\right)\\
=-\sum_{k=1}^{2M-2} \frac{i^k}{k(k+1)} \left[ B_{k+1}(\sigma)+\frac{k+1}{2}\sigma^k-\sigma^{k+1}\right] \frac 1{t^k}
+O\left( \frac{1}{|t|^{2M-1}}\right).
\end{multline*}
Therefore
\begin{multline} \label{ojk}
  \log \G(s) = \left(s-\frac 12\right)\log s - s +\frac 12 \log 2\pi  \\
   -\sum_{k=1}^{N-1} \frac{i^k}{k(k+1)} \left[ B_{k+1}(\sigma)+\frac{k+1}{2}\sigma^k-\sigma^{k+1}\right] \frac 1{t^k}
+O\left( \frac{1}{|t|^{N}}\right).
\end{multline}
A similar proof to Lemma \ref{bin} shows that, for $\sigma \in I$ and $t\neq 0$,
\begin{equation*}
  \log s =\log(it)+\log \left(1+\frac{\sigma}{it}\right) = \log(it)-\sum_{k=1}^{N-1} \frac{(i\sigma)^k}{k \cdot t^k} +O\left( \frac{1}{|t|^{N}}\right)
\end{equation*}
where the implied constant depends only on  $N \in \Z_{\gqs 1}$ and $I$.
Hence
\begin{equation} \label{ojk2}
  \left(s-\frac 12\right)\log s =(s-1/2)\log(it)+\sigma+\sum_{k=1}^{N-1}  \left( \frac it \right)^k \left[ \frac{\sigma^k}{2k} -\frac{\sigma^{k+1}}{k(k+1)}\right] +O\left( \frac{1}{|t|^{N}}\right).
\end{equation}
Inserting \e{ojk2} into \e{ojk} completes the proof.
\end{proof}

Since $-2i \vartheta(s) =(s-1/2)\log \pi+\log \G((1-s)/2)-\log \G(s/2)$ we easily now obtain

\begin{cor} \label{rsthe}
Suppose $s\in \C$ satisfies $\sigma \in I$ and $t \neq 0$. Then
\begin{multline*}
  i \vartheta(s) = \left( \frac s2-\frac 14\right) \log \frac{|t|}{2\pi} -\frac{i t}2 -\sgn(t)\frac{i \pi}8\\
   -\sum_{n=1}^{N-1} \left(\frac{2i}t \right)^n \left[\frac{ B_{n+1}(\sigma/2) +(-1)^{n+1} B_{n+1}((1-\sigma)/2) }{2n(n+1)}\right] +O\left( \frac 1{|t|^N}\right)
\end{multline*}
for an implied constant depending only on $I$  and $N\in \Z_{\gqs 1}$.
\end{cor}

\begin{cor} \label{rsthe2}
Suppose  $t \neq 0$. Then
\begin{equation*}
   \vartheta(1/2+i t) =  \frac t2 \log \frac{|t|}{2\pi} -\frac{ t}2 -\sgn(t)\frac{ \pi}8
   -\sum_{n=1}^{N-1} \frac{ (-4)^{n-1} B_{2n}(1/4)}{(2n-1)n \cdot t^{2n-1}}  +O\left( \frac 1{|t|^{2N-1}}\right)
\end{equation*}
for an implied constant depending only on  $N\in \Z_{\gqs 1}$.
\end{cor}

Corollary \ref{rsthe2} agrees with \cite[Satz 4.2.3(b)]{Ga79} as $B_{2n}(1/4)=2^{-2n}(2^{1-2n}-1)(-1)^{n+1}|B_{2n}|$.

\subsection{Bell polynomials}

The Bell polynomials give a convenient and explicit way to express the series coefficients we need. They complement the methods of Riemann and Siegel based on generating functions and recursions.

Let $p_1, p_2, p_3, \dots$ be any sequence of complex numbers.
The {\em partial ordinary Bell polynomials} $\hat{\mathcal B}_{i,j}$ give the coefficients of powers of the formal  series $p_1 x +p_2 x^2+ p_3 x^3+ \cdots $. With $j \in \Z_{\gqs 0}$,  the generating function definition is given by
\begin{equation} \label{pobell2}
    \left( p_1 x +p_2 x^2+ p_3 x^3+ \cdots \right)^j = \sum_{i=j}^\infty \hat{\mathcal B}_{i,j}(p_1, p_2, p_3, \dots) x^i.
\end{equation}
Clearly $\hat{\mathcal B}_{i,0}(p_1, p_2, p_3,  \dots) = \delta_{i,0}$.
We have the formulas
\begin{equation} \label{pobell}
    \hat{\mathcal B}_{i,j}(p_1, p_2, p_3, \dots) = \sum_{\substack{1\ell_1+2 \ell_2+ 3\ell_3+\dots = i \\ \ell_1+ \ell_2+ \ell_3+\dots = j}}
    \frac{j!}{\ell_1! \ell_2! \ell_3! \cdots } p_1^{\ell_1} p_2^{\ell_2} p_3^{\ell_3} \cdots
\end{equation}
from \cite[Sect. 3.3]{Com} where the sum is over all possible $\ell_1$, $\ell_2$, $\ell_3, \dots \in \Z_{\gqs 0}$, or
\begin{equation} \label{pobell3}
    \hat{\mathcal B}_{i,j}(p_1, p_2, p_3, \dots)= \sum_{n_1+n_2+\dots + n_j = i}
    p_{n_1}p_{n_2} \cdots p_{n_j} \qquad \qquad (j \gqs 1)
\end{equation}
 where the sum is over all possible $n_1$, $n_2,  \dots \in \Z_{\gqs 1}$. See the discussion and references in \cite[Sect. 7]{OSper} for how these Bell polynomials are used in the saddle-point method.

For $j \gqs 1$ we see from \e{pobell3} that $\hat{\mathcal B}_{i,j}(p_1, p_2, p_3, \dots)$ is a polynomial in $p_1, p_2, \dots, p_{i-j+1}$ of homogeneous degree $j$ with positive integer coefficients. For instance,
\begin{equation*}
  \hat{\mathcal B}_{8,4}(p_1, p_2, p_3, \dots) = p_2^4+12 p_1 p_2^2 p_3 + 6 p_1^2 p_3^2
+12 p_1^2 p_2 p_4 +4 p_1^3 p_5.
\end{equation*}
These polynomials may be expressed in terms of the closely related {\em partial  Bell polynomials} which are included in the  Mathematica system, for example. For a simple bound (which we will not need), note that
\begin{equation*}
  \hat{\mathcal B}_{i,j}(1,1,1, \dots)= \binom{i-1}{j-1} \qquad \qquad (j \gqs 1)
\end{equation*}
and hence, if $|p_n|\lqs Q$ for all $n\in \Z_{\gqs 1}$, then
\begin{equation*}
  \left|\hat{\mathcal B}_{i,j}(p_1, p_2, p_3,  \dots)\right| \lqs  \binom{i-1}{j-1}Q^j \qquad \qquad (j \gqs 1).
\end{equation*}

Based on the terms in Corollary \ref{rsthe} we make the definitions
\begin{align}
  f_n(\sigma) & :=\frac{B_{n+1}(\sigma/2)+(-1)^{n+1} B_{n+1}((1-\sigma)/2)}{2n(n+1)}, \label{fns}\\
 u_m(\sigma) & := (-2)^m \sum_{k=0}^m \frac 1{k!} \hat{\mathcal B}_{m,k}(f_1(\sigma),f_2(\sigma),\dots). \label{ums}
\end{align}
Then $f_n(\sigma)$ is a polynomial of degree $n+1$ with rational coefficients.
Hence,  $\hat{\mathcal B}_{m,k}(f_1(\sigma),f_2(\sigma),\dots)$ has degree at most $m+k$ by \e{pobell3}. It follows that $u_m(\sigma)$ is a polynomial  with rational coefficients. Its degree is exactly $2m$ since it may be checked that the coefficient of $u^{2m}$ in $u_m(\sigma)$ is $(-1)^m/(4^m m!)$. We have for example $u_0(\sigma)=1$ and
\begin{equation*}
   u_1(\sigma)=(-1+6\sigma-6\sigma^2)/24, \qquad u_2(\sigma)=(1 + 36 \sigma - 96 \sigma^2 + 24 \sigma^3 + 36 \sigma^4)/1152.
\end{equation*}
The next result requires the finite version of \e{pobell2}:
\begin{equation} \label{bellf}
    \left( p_1 x +p_2 x^2+ \cdots +p_r x^r\right)^j = \sum_{i=j}^{r j} \hat{\mathcal B}_{i,j}(p_1, p_2, \dots ,p_r,0,0, \dots) x^i.
\end{equation}

\begin{theorem} \label{rsu2}
Suppose $s\in \C$ satisfies $\sigma \in I$ and $t\gqs \epsilon >0$. Then
\begin{equation} \label{umexp}
  \exp\left( \left( \frac s2-\frac 14\right) \log \frac t{2\pi} -\frac{i t}2 -\frac{i \pi}8 - i \vartheta(s)\right) = \sum_{m=0}^{L-1} \frac{u_m(\sigma)}{(i t)^m}  +O\left( \frac 1{t^{L}}\right)
\end{equation}
for an implied constant depending only on $I$, $\epsilon$  and $L\in \Z_{\gqs 0}$.
\end{theorem}
\begin{proof}
We first note that for all $z\in \C$ with $|z|\lqs T$, and with an implied constant depending only on $K\in \Z_{\gqs 0}$ and $T$, we have
\begin{equation}\label{exp}
  e^z=\sum_{k=0}^{K-1} \frac{z^k}{k!} + O\left(|z|^K\right)
\end{equation}
by Taylor's Theorem with the usual remainder estimates.
Choose any $N\gqs 1$ and set
\begin{equation*}
  z_t:=  \left( \frac s2-\frac 14\right) \log \frac t{2\pi} -\frac{i t}2 -\frac{i \pi}8 - i \vartheta(s)
-\sum_{n=1}^{N-1} \left(\frac{2i}t \right)^n f_n(\sigma)
\end{equation*}
for any $s$ satisfying the conditions of the theorem. Then by Corollary \ref{rsthe} there is a constant $C_{I,N}$ so that
\begin{equation*}
  |z_t| \lqs C_{I,N}/t^N \lqs C_{I,N}/\epsilon^N.
\end{equation*}
Using \e{exp} with $T=C_{I,N}/\epsilon^N$ and $K=1$ we obtain $e^{z_t} = 1 +O(1/t^N)$, so that
\begin{align}
  \exp\left( \left( \frac s2-\frac 14\right) \log \frac t{2\pi} -\frac{i t}2 -\frac{i \pi}8 - i \vartheta(s)\right)
& = \exp\left( \sum_{n=1}^{N-1} \left(\frac{2i}t \right)^n f_n(\sigma) \right)\left( 1 +O\left(\frac 1{t^N}\right)\right)\label{notag}\\
& = \exp\left( \sum_{n=1}^{N-1} \left(\frac{2i}t \right)^n f_n(\sigma) \right) +O\left(\frac 1{t^N}\right). \label{sing}
\end{align}
When $N=1$ we mean $1+O(1/t)$ on the right of \e{notag} and \e{sing}.
For $N\gqs 2$, \e{sing} follows from \e{notag} by using that
\begin{equation} \label{exp2}
  \sum_{n=1}^{N-1} \left(\frac{2i}t \right)^n f_n(\sigma) = O\left(\frac 1{t}\right) \quad \implies \quad \exp\left( \sum_{n=1}^{N-1} \left(\frac{2i}t \right)^n f_n(\sigma) \right) = O\left(1\right).
\end{equation}
It is also true by \e{exp} and the left bound in \e{exp2} that
\begin{equation} \label{exp3}
  \exp\left( \sum_{n=1}^{N-1} \left(\frac{2i}t \right)^n f_n(\sigma) \right) = \sum_{k=0}^{K-1} \frac{1}{k!} \left( \sum_{n=1}^{N-1} \left(\frac{2i}t \right)^n f_n(\sigma) \right)^k +O\left(\frac 1{t^K}\right).
\end{equation}
By \e{bellf}, the sum on the right is
\begin{multline} \label{hum}
  \sum_{k=0}^{K-1} \frac{1}{k!} \left( \sum_{n=1}^{N-1} \left(\frac{2i}t \right)^n f_n(\sigma) \right)^k
= \sum_{k=0}^{K-1} \frac{1}{k!} \sum_{m=k}^{(N-1)k} \hat{\mathcal B}_{m,k}(f_1(\sigma),  \dots ,f_{N-1}(\sigma),0,0, \dots) \left(\frac{2i}t \right)^m \\
    = \sum_{m=0}^{(N-1)(K-1)}  \left(\frac{-2}{i t} \right)^m \times
\sum_{k=0}^{\min(m,K-1)} \frac{1}{k!} \hat{\mathcal B}_{m,k}(f_1(\sigma),  \dots ,f_{N-1}(\sigma),0,0, \dots).
\end{multline}
Recall that $\hat{\mathcal B}_{m,k}(f_1(\sigma),  \dots ,f_{N-1}(\sigma),0,0, \dots)$ just requires the first $m-k+1$ terms of the sequence $f_1(\sigma),  \dots ,f_{N-1}(\sigma),0,0, \dots$ and so will not use the $0$ terms if $m-k+1\lqs N-1$. Therefore we may write \e{hum} as
\begin{equation} \label{exp4}
  \sum_{m=0}^{N-2}  \left(\frac{-2}{i t} \right)^m \times
\sum_{k=0}^{m} \frac{1}{k!} \hat{\mathcal B}_{m,k}(f_1(\sigma), f_2(\sigma), \dots ) +O\left(\frac 1{t^{N-1}}\right)
\end{equation}
if $N-2\lqs K-1$. Assembling \e{sing}, \e{exp3} and \e{exp4} yields
\begin{multline*}
  \exp\left( \left( \frac s2-\frac 14\right) \log \frac t{2\pi} -\frac{i t}2 -\frac{i \pi}8 - i \vartheta(s)\right)\\
 = \sum_{m=0}^{N-2}  \left(\frac{-2}{i t} \right)^m \times
\sum_{k=0}^{m} \frac{1}{k!} \hat{\mathcal B}_{m,k}(f_1(\sigma), f_2(\sigma), \dots ) +O\left(\frac 1{t^{N-1}}\right) +O\left(\frac 1{t^N}\right) +O\left(\frac 1{t^K}\right)
\end{multline*}
for $K = N-1$. Letting $L=N-1$ in \e{umexp} completes the proof.
\end{proof}

\begin{cor} \label{thet}
Suppose $s\in \C$ satisfies $\sigma \in I$ and $t\gqs \epsilon >0$. Then
\begin{equation} \label{et}
  \exp\left( i \vartheta(s)\right) = O\left( t^{\sigma/2-1/4}\right), \qquad \exp\left(  - i \vartheta(s)\right) = O\left( t^{-\sigma/2+1/4}\right)
\end{equation}
for implied constants depending only on $I$   and $\epsilon$.
\end{cor}
\begin{proof}
By Theorem \ref{rsu2} with $L=1$,
\begin{equation} \label{rrr}
  \exp\left( \left( \frac s2-\frac 14\right) \log \frac t{2\pi} -\frac{i t}2 -\frac{i \pi}8 - i \vartheta(s)\right) = 1  +O\left( \frac 1{t}\right).
\end{equation}
It follows simply that the reciprocal of the left side has the same bound:
\begin{equation} \label{rrr2}
  \exp\left( -\left( \frac s2-\frac 14\right) \log \frac t{2\pi} +\frac{i t}2 +\frac{i \pi}8 + i \vartheta(s)\right) = 1  +O\left( \frac 1{t}\right).
\end{equation}
Multiplying both sides of \e{rrr2} by $\exp\left( \left( \frac s2-\frac 14\right) \log \frac t{2\pi} -\frac{i t}2 -\frac{i \pi}8\right)$ and bounding gives the left estimate in \e{et}. The right estimate is similar, manipulating \e{rrr}.
\end{proof}

It should be possible to replace the restriction $t\gqs \epsilon>0$ in Theorem \ref{rsu2} and Corollary \ref{thet} with just $t>0$. This would require a more careful treatment in Proposition \ref{logg} when $|t|<1$. Our applications will only require $t\gqs 2\pi$ in any case.

Set
\begin{equation*}
  g_n(\sigma):=-\frac{B_{n+1}(\sigma)}{n(n+1)} \qquad \text{and} \qquad
 \g_m(\sigma)  :=  i^m \sum_{k=0}^m \frac 1{k!} \hat{\mathcal B}_{m,k}(g_1(\sigma),g_2(\sigma),\dots)
\end{equation*}
so that, for instance, $\g_0(\sigma) = 1$ and
\begin{equation*}
   \g_1(\sigma) = (-1 + 6 \sigma - 6 \sigma^2)i/12, \qquad \g_2(\sigma) =  (-1 + 36 \sigma - 120 \sigma^2 + 120 \sigma^3 - 36 \sigma^4)/288.
\end{equation*}
Then a similar proof to that of Theorem \ref{rsu2}, using Proposition \ref{logg}, gives the asymptotics of the $\G$ function in vertical strips:
\begin{prop} \label{stir2}
Suppose $s\in \C$ satisfies $\sigma \in I$ and $t\gqs \epsilon >0$. Then
\begin{equation*} 
  \G(s) = \sqrt{2\pi}  \exp\left( \frac{\pi i s}{2}-i t -\frac{\pi i}{4}\right) t^{s-1/2} \left(\sum_{m=0}^{L-1} \frac{\g_m(\sigma)}{t^m}  +O\left( \frac 1{t^{L}}\right)\right)
\end{equation*}
for an implied constant depending only on $I$, $\epsilon$  and $L\in \Z_{\gqs 0}$.
\end{prop}

The power series coefficients $a_k(s)$ of $w(z,s)$ in \e{wzsd} and \e{wzs} may also be expressed in terms of Bell polynomials. For this we will need the identity
\begin{equation} \label{jude}
    \exp\left(u(p_1 x+p_2 x^2+\dots)\right) = \sum_{n=0}^\infty x^{n} \sum_{k=0}^n \hat{\mathcal B}_{n,k}(p_1,p_2,\dots) \frac{u^k}{k!}.
\end{equation}
 Define
\begin{equation} \label{dmr}
  d_{m,r}(\sigma) := \sum_{n=r}^m \frac{\hat{\mathcal B}_{n,r}(\frac 13,-\frac 14,\frac 15,\dots)}{r!} \sum_{k=0}^{m-n} \frac{\hat{\mathcal B}_{m-n,k}(1,-\frac 12,\frac 13,\dots)}{k!}(\sigma-1)^k.
\end{equation}

\begin{prop} \label{akbell}
For all $k\in \Z_{\gqs 0}$ we have
\begin{equation} \label{akexp}
  a_k(s)=  \sum_{r=0}^{\lfloor k/3\rfloor} i^r  \cdot d_{k-2r,r}(\sigma)  \cdot  t^{r-k/2}.
\end{equation}
\end{prop}
\begin{proof}
Expanding the logarithm in  \e{wzsd} into its power series and then employing \e{jude} produces
\begin{multline*}
  w(z,s) = \exp\left( (\sigma-1) \sum_{j=1}^\infty \frac{(-1)^{j+1}}{j} \left(\frac z{\sqrt t}\right)^j \right) \exp\left( i z^2  \sum_{j=1}^\infty \frac{(-1)^{j+1}}{j+2} \left(\frac z{\sqrt t}\right)^j \right) \\
  = \left( \sum_{n=0}^\infty \left( \frac{z}{\sqrt t}\right)^n \sum_{k=0}^n \hat{\mathcal B}_{n,k}( 1,{\textstyle -\frac 12}, {\textstyle \frac 13},\dots) \frac{(\sigma-1)^k}{k!}\right)
\left( \sum_{m=0}^\infty \left( \frac{z}{\sqrt t}\right)^m \sum_{r=0}^m \hat{\mathcal B}_{m,r}({\textstyle \frac 13},{\textstyle -\frac 14}, {\textstyle \frac 15},\dots) \frac{(i z^2)^r}{r!}\right)\\
=
\sum_{h=0}^\infty \left( \frac{z}{\sqrt t}\right)^h \sum_{r=0}^h (i z^2)^r \sum_{m=r}^h \frac{\hat{\mathcal B}_{m,r}(\frac 13,-\frac 14,\frac 15,\dots)}{r!} \sum_{k=0}^{h-m} \frac{\hat{\mathcal B}_{h-m,k}(1,-\frac 12,\frac 13,\dots)}{k!}(\sigma-1)^k.
\end{multline*}
Therefore
\begin{align*}
  w(z,s)  & = \sum_{h=0}^\infty t^{-h/2} \sum_{r=0}^h i^r  \cdot d_{h,r}(\sigma) \cdot  z^{h+2r} \\
   & = \sum_{k=0}^\infty z^k \sum_{r=0}^{\lfloor k/3\rfloor} i^r  \cdot d_{k-2r,r}(\sigma)  \cdot  t^{r-k/2}.
\qedhere
\end{align*}
\end{proof}

For example $a_0(s)=1$, $a_1(s)=(\sigma-1)/t^{1/2}$ and
\begin{equation*}
   a_2(s)=(\sigma^2-3\sigma+2)/(2t), \qquad a_3(s)=(\sigma^3-6\sigma^2+11\sigma+2i t -6)/(6t^{3/2}).
\end{equation*}
Siegel gave $a_k(s)$ in terms of the recursion \e{akrec}. The advantage of Proposition \ref{akbell} is that it gives explicit formulas for the coefficients of the powers of $t$ in $a_k(s)$. These formulas will be needed in the next section.

\section{Proof of most parts of the main theorem}

Our goal in this section is the next result.

\begin{theorem} \label{almost}
Recall the statement of Theorem \ref{unsym-rs}. This statement, with the change that the implied constant in \e{rrss} may also depend on  $\lambda$ when $\lambda <1$, is true.
\end{theorem}
\begin{proof}
We begin with Theorem \ref{hat} and  multiply both sides of \e{fab} by $e^{i\vartheta(s)}$.  Corollary \ref{thet}
gives the estimate $e^{i\vartheta(s)}=O(t^{\sigma/2-1/4})$. It is convenient to abbreviate the inner sum in \e{fab} as
\begin{equation} \label{ckl}
  c_k(\lambda) := \sum_{r=0}^k \binom{k}{r} G^{(r)}(a\lambda^{-1}+b \lambda  ; \lambda^2 )
  \frac{e^{\pi i(k-3r)/4}}{2^{k-r}(2\pi)^{r/2}}H_{k-r} \left(  \omega_\lambda\right).
\end{equation}
Thus we have shown the following. Let $\sigma \in I$ and $t=2\pi\alpha \beta$  for real numbers $\alpha,\beta\gqs 1$.  Then for all $N \in \Z_{\gqs 0}$ we have
\begin{multline} \label{fabr}
  R(s;\alpha,\beta)
  =
  e^{i\vartheta(s)} (-1)^{\lfloor \alpha \rfloor \lfloor \beta \rfloor} \frac {(2\pi)^s e^{\pi i s/2}}{\G(s)(e^{2\pi i s}-1)}
 \\
   \times
       \exp\left( \frac{\pi i}{2} \Bigl[2a\beta-2b\alpha+ a^2 \lambda^{-2}  - b^2 \lambda^2\Bigr]\right)
\exp\left( \left( \frac s2-\frac 14\right) \log \frac t{2\pi} -\frac{i t}2 -\frac{i \pi}8\right)
\\
  \times \left(\frac{2\pi}t\right)^{1/4}\lambda^{1/2-s} \sum_{k=0}^{N-1} a_k(s) \cdot
  c_k(\lambda)
  +O\left( \frac{\lambda^{1-\sigma}t^{-1/4}}{t^{N/6}}\right).
\end{multline}
The implied constant in \e{fabr} depends only on $I$, $N$ and $\lambda$. If $\lambda \gqs 1$ then the implied constant is independent of $\lambda$.
We may simplify the initial terms on the right of \e{fabr} by noting that
\begin{equation*}
  e^{-2i\vartheta(s)} =\chi(s) = \frac{(2\pi)^s}{2 \cos\left(\pi s/2\right) \G(s)}.
\end{equation*}
Then
\begin{align}
  e^{i\vartheta(s)}\frac {(2\pi)^s e^{\pi i s/2}}{\G(s)(e^{2\pi i s}-1)} & =
   e^{i\vartheta(s)} \frac{e^{\pi i s/2}}{(e^{2\pi i s}-1)} 2 \cos(\pi s/2) e^{-2i\vartheta(s)}  \label{usop}\\
   & = e^{-i\vartheta(s)} \frac{e^{\pi i s}+1}{e^{2\pi i s}-1}
   = \frac{ e^{-i\vartheta(s)}}{e^{\pi i s}-1}= - e^{-i\vartheta(s)}\left(1+O(e^{-\pi t}) \right). \notag
\end{align}
So replacing the left side of \e{usop} with $ - e^{-i\vartheta(s)}$ in \e{fabr}  introduces an error of size
\begin{equation}\label{teni}
  \left| e^{-i\vartheta(s)} \left(\frac t{2\pi} \right)^{(s-1)/2} \lambda^{1/2-s} \sum_{k=0}^{N-1} a_k(s) \cdot
  c_k(\lambda) \right| e^{-\pi t}.
\end{equation}
By Corollary \ref{thet}, $e^{-i\vartheta(s)} =O(t^{-\sigma/2+1/4})$.
We have $a_k(s) = O(1)$ by \e{aksbnd}.
With Corollary \ref{lamb} and the fact that $H_{n}(x)$ has degree $n$, it follows from \e{ckl} that
\begin{equation}\label{www}
  c_k(\lambda) \ll \lambda^{k+1/2}+\lambda^{-k-1/2}.
\end{equation}
Putting these estimates together shows that \e{teni} is
\begin{equation*}
  O\left(\lambda^{1/2-\sigma}(\lambda^{N-1/2}+ \lambda^{-N+1/2}) t^{-1/4} e^{-\pi t} \right).
\end{equation*}

Our results so far have established the next estimate (replacing $N$ with $M$).
\begin{prop} Let $\sigma \in I$ and $t=2\pi\alpha \beta$  for real numbers $\alpha,\beta\gqs 1$.  Then for all $M \in \Z_{\gqs 0}$ we have
\begin{multline} \label{nak}
  R(s;\alpha,\beta) =
   (-1)^{\lfloor \alpha \rfloor \lfloor \beta \rfloor+1}  \exp\left( \frac{\pi i}{2} \Bigl[2a\beta-2b\alpha+ a^2 \lambda^{-2}  - b^2 \lambda^2\Bigr]\right)\\
   \times
      \exp\left( \left( \frac s2-\frac 14\right) \log \frac t{2\pi} -\frac{i t}2 -\frac{i \pi}8 - i \vartheta(s)\right)
      \left(\frac{2\pi}t\right)^{1/4}\lambda^{1/2-s}
      \sum_{k=0}^{M-1} a_k(s) \cdot
  c_k(\lambda) \\
+O\left( \frac{\lambda^{1/2-\sigma}}{t^{1/4}}\left(\lambda^{M-1/2}+ \lambda^{-M+1/2}\right)  e^{-\pi t} + \frac{\lambda^{1-\sigma}}{t^{M/6+1/4}}\right).
\end{multline}
The implied constant in \e{nak} depends only on $I$, $M$ and $\lambda$. If $\lambda \gqs 1$ then the implied constant is independent of $\lambda$.
\end{prop}

The proof of Theorem \ref{almost} continues by inserting \e{umexp} and \e{akexp} into \e{nak} to obtain the desired asymptotic expansion  in decreasing powers of $t$. An argument similar to the one bounding \e{teni} shows that the total error introduced from the error term in \e{umexp} is
\begin{equation}\label{erru}
  O\left(\lambda^{1/2-\sigma}\left(\lambda^{M-1/2}+ \lambda^{-M+1/2}\right) t^{-1/4-L}  \right).
\end{equation}
Ignoring the constant and modulus $1$ pieces of \e{nak} for the moment, we have
\begin{multline} \label{tyu}
  \frac{\lambda^{1/2-s}}{  t^{1/4}}\left(  \sum_{m=0}^{L-1} \frac{u_m(\sigma)}{(i t)^m} \right)\sum_{k=0}^{M-1} c_k(\lambda)  \left( \sum_{j=0}^{\lfloor k/3\rfloor}  i^j \cdot d_{k-2j,j}(\sigma)  \cdot  \frac{1}{t^{k/2-j}}\right) \\
  = \frac{\lambda^{1/2-s}}{  t^{1/4}}\sum_{n=0}^{M+2L-3} \frac 1{t^{n/2}} \sum_{\substack{ m,k,j, \\2m+k-2j=n}} c_k(\lambda)  \cdot  i^{j-m}
   \cdot d_{k-2j,j}(\sigma)  \cdot u_m(\sigma)\\
   = \frac{\lambda^{1/2-s}}{  t^{1/4}}\sum_{n=0}^{M+2L-3} \frac 1{t^{n/2}} \sum_{k} c_k(\lambda) \cdot i^{(k-n)/2} \sum_j
    d_{k-2j,j}(\sigma)  \cdot u_{j+(n-k)/2}(\sigma)
\end{multline}
where in the last line we are summing over all $k$ and $j$ such that
\begin{gather} \label{kj}
  k \equiv n \bmod 2, \quad 0\lqs k \lqs M-1, \quad 2-2L+n \lqs k \lqs 3n, \\
  0\lqs j \lqs k/3, \quad (k-n)/2 \lqs j \lqs (k-n)/2+L-1. \label{kj2}
\end{gather}
The natural ranges of $k$ and $j$ are $0\lqs k \lqs 3n$ and $ \max(0,(k-n)/2)\lqs j \lqs k/3$, but for large $n$ these ranges become truncated. We may choose $N$ small enough in relation to $M$ and $L$ so that, for $0\lqs n\lqs N-1$, the ranges of $k$ and $j$ are not truncated. This requires
\begin{equation}\label{bml}
  M\gqs 3N+1, \qquad L \gqs N/2+1.
\end{equation}
The size of the remaining part of the sum \e{tyu} with $N\lqs n \lqs M+2L-3$ is  $O(t^{-N/2-1/4})$ in $t$ (see the next lemma) and we also require that  the error $O(t^{-M/6-1/4})$ in \e{nak} and the error $O(t^{-L-1/4})$ in \e{erru} are both less than this. This requires $M \gqs 3N$ and $L\gqs N/2$ and so is already ensured by \e{bml}. Given $N$, we therefore choose $M=3N+1$ and $L=\lceil N/2 \rceil+1$.

\begin{lemma}
We have
\begin{multline} \label{erru2}
  \frac{\lambda^{1/2-s}}{  t^{1/4}}\sum_{n=N}^{3N+2\lceil N/2 \rceil} \frac 1{t^{n/2}} \sum_{k} c_k(\lambda) \cdot i^{(k-n)/2} \sum_j
    d_{k-2j,j}(\sigma)  \cdot u_{j+(n-k)/2}(\sigma) \\
= O\left( \frac{\lambda^{1/2-\sigma}}{ t^{N/2+1/4}} \left( \lambda^{3N+1/2}+ \lambda^{-3N-1/2}\right)\right)
\end{multline}
where the indices $k$ and $j$ sum over the ranges \e{kj} and \e{kj2} for $M=3N+1$ and $L=\lceil N/2 \rceil+1$. The implied constant depends only on $N$ and $I$.
\end{lemma}
\begin{proof}
 By \e{kj}, the largest $k$ appearing in the sum is $3N$. Therefore $c_k(\lambda)$ is always $\ll \lambda^{3N+1/2}+ \lambda^{-3N-1/2}$ with \e{www}. The other bounds are clear.
\end{proof}

For $n\equiv k \bmod 2$ let
\begin{equation} \label{qnk}
  q_{n,k}(\sigma) := \sum_{j=\max(0,(k-n)/2)}^{\lfloor k/3\rfloor} d_{k-2j,j}(\sigma)  \cdot u_{j+(n-k)/2}(\sigma).
\end{equation}
With the definitions \e{dmr} and \e{ums} it is clear that $q_{n,k}(\sigma)$ is a polynomial in $\sigma$ with rational coefficients. Since $d_{m,r}(\sigma)$ has degree $m-r$ and $u_m(\sigma)$ has degree $2m$ it follows that $q_{n,k}(\sigma)$ has degree at most $n$.

For our choice of $M$ and $L$, the error in \e{erru2} is larger then the error terms in \e{nak} and \e{erru}. Therefore we have shown that
\begin{multline} \label{nak2}
  R(s;\alpha,\beta) =
   (-1)^{\lfloor \alpha \rfloor \lfloor \beta \rfloor+1}  \exp\left( \frac{\pi i}{2} \Bigl[2a\beta-2b\alpha+ a^2 \lambda^{-2}  - b^2 \lambda^2\Bigr]\right) (2\pi)^{1/4}
   \\
      \times  \frac{\lambda^{1/2-s}}{ t^{1/4}}\sum_{n=0}^{N-1} \frac 1{t^{n/2}}
\sum_{\substack{0\lqs k \lqs 3n\\ k \equiv n \bmod 2}}  i^{(k-n)/2} q_{n,k}(\sigma)
  \sum_{r=0}^k \binom{k}{r}  G^{(r)}(a\lambda^{-1}+b \lambda  ; \lambda^2 )
  \frac{e^{\pi i(k-3r)/4}}{2^{k-r}(2\pi)^{r/2}}H_{k-r} \left( \omega_\lambda\right)\\
  + O\left( \frac{\lambda^{1/2-\sigma}}{ t^{N/2+1/4}} \left( \lambda^{3N+1/2}+ \lambda^{-3N-1/2}\right)\right).
\end{multline}
The sums over $k$ and $r$ in \e{nak2}, after interchanging, are
\begin{equation*}
  \sum_{r=0}^{3n} \frac{G^{(r)}(a\lambda^{-1}+b \lambda  ; \lambda^2 )}{(2\pi)^{r/2}} e^{\pi i(n-3r)/4}
  \sum_{\substack{r\lqs k \lqs 3n\\ k \equiv n \bmod 2}} \binom{k}{r}   \frac{(-1)^{(n-k)/2}}{2^{k-r}} \cdot q_{n,k}(\sigma) \cdot H_{k-r} \left( \omega_\lambda\right).
\end{equation*}
Recall that $\omega_\lambda = e^{-\pi i/4} \sqrt{\frac \pi 2} ( a \lambda^{-1} - b \lambda)$. Write the inner piece as
\begin{equation*}
  P_{n,3n-r}(x,\sigma)  := e^{\pi i(n-3r)/4} \sum_{\substack{r\lqs k \lqs 3n\\ k \equiv n \bmod 2}} \binom{k}{r}   \frac{(-1)^{(n-k)/2}}{2^{k-r}} \cdot  q_{n,k}(\sigma) \cdot H_{k-r} \left( e^{-\pi i/4} x \right).
\end{equation*}
Then
\begin{equation}\label{mell}
  P_{n,k}(x,\sigma) =  e^{3\pi i k/4}  \sum^{\lfloor k/2 \rfloor}_{\ell=0} \binom{3n-2\ell}{3n-k} \frac{(-1)^{n+\ell}}{2^{k-2\ell}}
   \cdot q_{n,3n-2\ell}(\sigma) \cdot H_{k-2\ell} \left(e^{-\pi i/4} x\right).
\end{equation}
Clearly $ P_{n,k}(x,\sigma)$ is a polynomial in $x$ and $\sigma$ with  degree at most $k$ in $x$. A short calculation finds that the coefficient of $x^{k}$ is $i^k \binom{3n}{k}/((-3)^n n!)$ and so the degree is exactly $k$.
The complete construction of $P_{n,k}(x,\sigma)$ is repeated for convenience in \e{poly}. This  finishes the proof of Theorem \ref{almost}.
\end{proof}

 If $\lambda <1$ then the implied constant in Theorem \ref{almost} may have extra $\lambda$ dependence; this can be traced back to Proposition \ref{longj}. We will use the symmetry \e{refl} to fix this issue in the next section and complete the proof of Theorem \ref{unsym-rs}.

\section{The polynomials $P_{n,k}(x,\sigma)$} \label{pnry}
\subsection{A functional equation}
Recall the Bernoulli,  Hermite and Bell polynomials from \e{behe} and \e{pobell2}. Assembling our results, we may give a complete description of $P_{n,k}(x,\sigma)$ in terms of these polynomials as follows. In \e{fns}, \e{ums} and  \e{dmr}  we defined
\begin{subequations} \label{poly}
\begin{align}
  f_n(\sigma) & :=\frac{B_{n+1}(\sigma/2)+(-1)^{n+1} B_{n+1}((1-\sigma)/2)}{2n(n+1)}, \label{poly:a}\\
 u_m(\sigma) & := (-2)^m \sum_{k=0}^m \frac 1{k!} \hat{\mathcal B}_{m,k}(f_1(\sigma),f_2(\sigma),\dots),\label{poly:b}\\
d_{m,r}(\sigma) & := \sum_{n=r}^m \frac{\hat{\mathcal B}_{n,r}(\frac 13,-\frac 14,\frac 15,\dots)}{r!} \sum_{k=0}^{m-n} \frac{\hat{\mathcal B}_{m-n,k}(1,-\frac 12,\frac 13,\dots)}{k!}(\sigma-1)^k.\label{poly:c}
\end{align}
Rearranging \e{qnk} we may set
\begin{equation} \label{poly:d}
  q_{n,3n-2\ell}(\sigma) := \sum_{m=\max(0,\ell-n)}^{\lfloor \ell/3\rfloor} u_m(\sigma) \cdot d_{n-2m,n-\ell+m}(\sigma)
\end{equation}
for $0\lqs \ell \lqs \lfloor 3n/2\rfloor$.
Then as we  saw with \e{mell},
\begin{equation} \label{poly:e}
  P_{n,k}(x,\sigma) =  e^{3\pi i k/4}  \sum^{\lfloor k/2 \rfloor}_{\ell=0} \binom{3n-2\ell}{3n-k} \frac{(-1)^{n+\ell}}{2^{k-2\ell}}
   \cdot q_{n,3n-2\ell}(\sigma) \cdot H_{k-2\ell} \left(e^{-\pi i/4} x\right).
\end{equation}
\end{subequations}

We next show that the polynomials $P_{n,k}(x,\sigma)$ obey a functional equation as $\sigma \to 1-\sigma$. It seems difficult to prove this  directly with \e{poly}; our proof is based on Theorem \ref{almost}.

\begin{theorem} \label{polss}
For all $x,\sigma \in \R$ and all $n,k\in \Z$ with $0\lqs k \lqs 3n$ we have
\begin{equation*}
  P_{n,k}(x,\sigma)=\overline{P_{n,k}(-x,1-\sigma)}.
\end{equation*}
\end{theorem}
\begin{proof}
Recall that $R(s;\alpha,\beta)$ is unchanged under the transformation $\mathcal T$ given in  \e{tra}. All the components of the right side of \e{rrss}, except for possibly $P_{n,k}$, are also unchanged under $\mathcal T$. For example
\begin{equation*}
\begin{array}{rcccl}
 \mathcal  T(\lambda^{1/2-s}) & = & \overline{(1/\lambda)^{1/2-(1-\overline{s})}} & = &\lambda^{1/2-s},\\
\mathcal T G^{(r)}(a\lambda^{-1}+b \lambda  ; \lambda^2 ) & = & \overline{G^{(r)}(b \lambda + a\lambda^{-1}  ; \lambda^{-2} )}
&  = & G^{(r)}(a\lambda^{-1}+b \lambda  ; \lambda^2 )
\end{array}
\end{equation*}
using \e{fft}. Hence, by  Theorem \ref{almost} we obtain
\begin{multline} \label{poi}
   \sum_{n=0}^{N-1} \frac 1{t^{n/2}}  \sum_{r=0}^{3n} \frac{G^{(r)}(a  \lambda^{-1} +b \lambda ;\lambda^2 )}{(2\pi)^{r/2}} \\
\times \left[P_{n,3n-r}\left(\frac{\sqrt{\pi}}{\sqrt{2}}(a \lambda^{-1} - b \lambda),\sigma\right)- \overline{P_{n,3n-r}\left(\frac{\sqrt{\pi}}{\sqrt{2}}(b \lambda - a \lambda^{-1}),1-\sigma\right)}\right]
=O\left( \frac {1}{t^{N/2}}\right)
\end{multline}
where the implied constant depends on $N\in \Z_{\gqs 0}$ and also on $\lambda$ and $\sigma$ which we assume are fixed. We choose $\lambda$ such that $\lambda^2$ is a rational $u/v$  with $(u,v)=1$. If we think of $\alpha$ varying then we have the dependent relations
\begin{equation} \label{dep}
  \beta= \frac{v}{u}\alpha, \quad t=2\pi \frac{v}{u}\alpha^2,  \quad a=\alpha-\lfloor \alpha \rfloor,  \quad b= \frac{v}{u}\alpha - \left\lfloor \frac{v}{u}\alpha \right\rfloor.
\end{equation}
Suppose $\alpha=\alpha_0$ has the corresponding $a$ and $b$ values $a_0$ and $b_0$, respectively. Then clearly $\alpha=\alpha_0+u$ will have the same $a$ and $b$ values. Hence, for $t$ taking values in the sequence $2\pi \frac{v}{u}(\alpha_0+ku)^2$ for integers $k$, the inner sum in \e{poi} is unchanged with $a=a_0$ and $b=b_0$. By letting $k$, and hence $t$, become arbitrarily large we obtain
\begin{equation*}
  \sum_{r=0}^{3n}  \frac{G^{(r)}(a  \lambda^{-1} +b \lambda ;\lambda^2 )}{(2\pi)^{r/2}}
\left[P_{n,3n-r}\left(\frac{\sqrt{\pi}}{\sqrt{2}}(a \lambda^{-1} - b \lambda),\sigma\right)- \overline{P_{n,3n-r}\left(\frac{\sqrt{\pi}}{\sqrt{2}}(b \lambda - a \lambda^{-1}),1-\sigma\right)}\right] =0
\end{equation*}
for $a=a_0$ and $b=b_0$ and $n \lqs N-1$. This follows since we may first show that the coefficient of $1/t^0$ must be $0$. Then the coefficient of $1/t^{1/2}$ must be $0$, etc. (If asymptotic expansions exist then they are unique.)

Let $f(\alpha):=a \lambda^{-1} + b \lambda$ and $g(\alpha):=a \lambda^{-1} - b \lambda$ where $a$ and $b$ depend on $\alpha$ as in \e{dep}. As we already saw, $f(\alpha)$ and $g(\alpha)$ both have period $u$. The next result shows which values they can take and we omit its elementary proof.

\begin{lemma} \label{fguv}
For $f(\alpha)$ and $g(\alpha)$ as defined above:
\begin{enumerate}
  \item The function $g(\alpha)$  is constant on $u+v-1$ non-empty intervals $[x_j,x_{j+1})$ which partition $[0,u)$.
  \item The values $g(\alpha)$ takes are
\begin{equation} \label{vals}
  \frac{\ell}{\sqrt{uv}} \quad \text{for integers $\ell$ with}  \quad 1-u \lqs \ell \lqs v-1.
\end{equation}
  \item On each interval $[x_j,x_{j+1})$ where $g(\alpha)$ is constant, the graph of $f(\alpha)$ is a line of slope $2/\lambda$.
\end{enumerate}
\end{lemma}

Fix $x$ as one of the $g(\alpha)$ values in \e{vals}. By part (iii) of the lemma we have
\begin{equation} \label{cohe}
  \sum_{r=0}^{3n}  \frac{G^{(r)}(w ;\lambda^2 )}{(2\pi)^{r/2}} \left[P_{n,3n-r}\left(\frac{\sqrt{\pi}}{\sqrt{2}}x ,\sigma\right)- \overline{P_{n,3n-r}\left(-\frac{\sqrt{\pi}}{\sqrt{2}}x,1-\sigma\right)}\right] =0
\end{equation}
for $w$ in some non-empty interval. Since each $G^{(r)}(w ;\lambda^2 )$ is a holomorphic function of $w$, it follows that \e{cohe} is true for all $w\in \C$.

The linear independence of the derivatives of $G$ shown in Proposition \ref{ffff} implies that
\begin{equation}\label{jhg}
  P_{n,k}\left(\frac{\sqrt{\pi}}{\sqrt{2}}x ,\sigma\right)- \overline{P_{n,k}\left(-\frac{\sqrt{\pi}}{\sqrt{2}}x,1-\sigma\right)}
\end{equation}
is $0$ for every $k$ with $0\lqs k \lqs 3n$. Hence the numbers \e{vals} give $u+v-1$ distinct zeros of the polynomial \e{jhg} in $x$. It has degree at most $3n$ in $x$. Returning to our choice of $\lambda^2=u/v$, we may choose a reduced fraction so that $u+v-1> 3n$. This gives too many zeros and so \e{jhg} is identically zero as required.
\end{proof}

\begin{proof}[Proof of Theorem \ref{unsym-rs}]
By Theorem \ref{almost}, we know that Theorem \ref{unsym-rs} is true for $\lambda \gqs 1$, so assume that $\lambda < 1$. Then by \e{refl} we have $R(s;\alpha,\beta) = \overline{R(1-\overline{s};\beta,\alpha)}$. We may apply Theorem \ref{almost} to $R(1-\overline{s};\beta,\alpha)$ and obtain an error that is independent of $\lambda$.  All the components of the right side of \e{rrss} are invariant under the transformation $\mathcal T$ in \e{tra}, including the $P_{n,k}$ term by Theorem \ref{polss}. In this way we obtain Theorem \ref{unsym-rs} when $\lambda < 1$. This completes our proof of the main theorem.
\end{proof}

\subsection{Formulas for the coefficients}
With the well-known formula for Hermite polynomials
\begin{equation} \label{hann}
  H_n(x)=n!\sum_{j=0}^{\lfloor n/2\rfloor} \frac{(-1)^j}{j!(n-2j)!}(2x)^{n-2j}
\end{equation}
we obtain from \e{mell}
\begin{align*}
  P_{n,k}(x,\sigma)  & = \sum^{\lfloor k/2 \rfloor}_{\ell=0} \frac{(3n-2\ell)!}{(3n-k)!} \frac{(-1)^{n+\ell}}{2^{k-2\ell}}
   \cdot q_{n,3n-2\ell}(\sigma) \sum^{\lfloor k/2 \rfloor}_{m=\ell} \frac{(-1)^{m-\ell} i^{k+m} (2x)^{k-2m}}{(m-\ell)!(k-2m)!}\\
& =  \frac{(-1)^n}{(3n-k)!}\sum^{\lfloor k/2 \rfloor}_{m=0} \frac{ i^{k+m} x^{k-2m}}{(-4)^m(k-2m)!}
\sum_{\ell=0}^m 4^\ell \frac{(3n-2\ell)!}{(m-\ell)!}\cdot   q_{n,3n-2\ell}(\sigma).
\end{align*}
For $0\lqs m\lqs \lfloor 3n/2 \rfloor$ put
\begin{equation} \label{snm}
  s_{n,m}(\sigma):=  \sum_{\ell=0}^m 4^\ell \frac{(3n-2\ell)!}{(m-\ell)!} \cdot q_{n,3n-2\ell}(\sigma)
\end{equation}
so that
\begin{equation}
  P_{n,k}(x,\sigma)  = \frac{(-1)^n i^{k} }{(3n-k)!} \sum^{\lfloor k/2 \rfloor}_{m=0} \frac{s_{n,m}(\sigma)}{(4 i)^m} \frac{ x^{k-2m}}{(k-2m)!}. \label{npp2}
\end{equation}

An easy calculation with \e{npp2} gives the next result,  showing how Theorem \ref{polss} may be interpreted at the level of the coefficients of $P_{n,k}(x,\sigma)$.

\begin{lemma} \label{cong}
Let $n$ be in $\Z_{\gqs 0}$. For all integers  $k$ with $0\lqs k \lqs 3n$, we have $P_{n,k}(x,\sigma)=\overline{P_{n,k}(-x,1-\sigma)}$ if and only if
\begin{equation*}
  s_{n,m}(\sigma) = (-1)^m s_{n,m}(1-\sigma)
\end{equation*}
for all integers $m$ satisfying $0 \lqs m \lqs \lfloor 3n/2 \rfloor$.
\end{lemma}

 Theorem \ref{polss} and Lemma \ref{cong} show in particular that $s_{n,m}(1/2) =0$ for $m$ odd. Hence $m$ may be assumed to be even in \e{npp2} when $\sigma=1/2$. We obtain
\begin{equation} \label{npp3}
   P_{n,k}(x,1/2) = \frac{(-1)^n i^{k} }{(3n-k)!} \sum^{\lfloor k/4 \rfloor}_{m=0} (-1)^m \frac{ s_{n,2m}(1/2) }{16^{m}} \frac{x^{k-4m}}{(k-4m)!}.
\end{equation}
In the case $x=0$ we see that  $P_{n,k}(0,1/2)$ is zero if $k \not\equiv 0 \bmod 4$. This explains why only every fourth derivative appears in the classical Riemann-Siegel formulas \e{riem}, \e{rrss40}.

The coefficients of the highest powers of $x$ in $P_{n,k}(x,\sigma)$ may be computed explicitly:
\begin{subequations} \label{polyc}
\begin{align}
  x^k: \qquad & \frac{(-1)^n i^k}{3^{n}n!}\binom{3n}{k}, \\
  x^{k-2}: \qquad & \frac{(-1)^n i^{k-1}}{3^{n-1}(n-1)!}\binom{3n-2}{k-2}(\sigma-1/2),\\
x^{k-4}: \qquad & \frac{(-1)^n i^{k}}{3^{n-2}(n-2)!}\binom{3n-4}{k-4}\left[ \frac{3n-1}{20}-\frac 12 (\sigma-1/2)^2\right],\\
x^{k-6}: \qquad & \frac{(-1)^n i^{k-1}}{3^{n-2}(n-2)!}\binom{3n-6}{k-6}(\sigma-1/2)
\left[\frac{9n^2-20n+9}{20}- \frac{n-2}{2}(\sigma-1/2)^2\right].
\end{align}
The general pattern continues with the next coefficient
\begin{multline}
  x^{k-8}: \qquad \frac{(-1)^n i^k}{3^{n-3}(n-3)!}\binom{3n-8}{k-8} \\
  \times \left[\frac{63n^2-141n+31}{5600}(3n-5)
-\frac{9n^2-28n+23}{40}(\sigma-1/2)^2+\frac{n-3}{8} (\sigma-1/2)^4\right].
\end{multline}
\end{subequations}
These calculations use \e{npp2}, \e{snm} and \e{poly}. Finding $d_{m,r}(\sigma)$ involves the Bell polynomials and   we used the  identities for $r\in \Z_{\gqs 0}$
\begin{align*}
  \hat{\mathcal B}_{r,r}(p_1, p_2, p_3, \dots) & = p_1^r, \\
  \hat{\mathcal B}_{r+1,r}(p_1, p_2, p_3, \dots) & = r p_1^{r-1} p_2,\\
\hat{\mathcal B}_{r+2,r}(p_1, p_2, p_3, \dots) & = \binom{r}{2} p_1^{r-2} p_2^2 + \binom{r}{1} p_1^{r-1} p_3,
\end{align*}
and so on. These follow from \e{pobell}. The algebra to obtain the coefficients in \e{polyc} was carried out with  Mathematica.

\section{Examples and numerical work} \label{numb}

\subsection{The case $\alpha=2\beta$}
The classical case of the  Riemann-Siegel formula has the lengths of the partial sums equal, so that $\alpha=\beta$ and $\lambda=1$. In Theorem \ref{unsym-rs}, the next simplest case has the length of one partial sum twice the other:
\begin{equation*}
  \alpha=2\beta \quad \implies \quad \lambda=\sqrt{2}, \quad \alpha= \sqrt{\frac t\pi}, \quad \beta= \frac 12 \sqrt{\frac t\pi}.
\end{equation*}
With $a$ and $b$ the fractional parts of $2\beta$ and $\beta$ we obtain
\begin{multline} \label{rrx}
  R(s;2\beta,\beta) =
   (-1)^{\lfloor 2\beta \rfloor \lfloor \beta \rfloor+1}
\exp\left(\pi i(a\beta-2b\beta+ a^2/4  - b^2)\right)
   \\
      \times   \left( \frac{2\pi}{t}\right)^{1/4}   \sum_{n=0}^{N-1} \frac {2^{1/4-s/2}}{t^{n/2}} \left[ \sum_{r=0}^{3n}
    \frac{G^{(r)}(a/\sqrt{2} +\sqrt{2} b  ;2 )}{(2\pi)^{r/2}} \cdot P_{n,3n-r}\left(\sqrt{\pi} (a/2 - b),\sigma\right)\right]
\\
+O\left( \frac {2^{(1-\sigma + 3N)/2}}{t^{N/2+1/4}}\right).
\end{multline}
As we saw in \e{gu2}, with $\theta_{2}(u):=u^2/2 -\sqrt{2} u-9/8$,
$$
G(u;2)  =\frac{-1}{2i \sin(\sqrt{2}\pi u)}\left[2^{1/4}e^{-\pi i \theta_2(u)} - 2^{-1/4}e^{\pi i \theta_2(u)}\left(1+i e^{\sqrt{2}\pi i u}\right) \right].
$$
It is easy to see that the polynomials $P_{n,3n-r}(x,\sigma)$ in \e{rrx} are only evaluated at $x=0$ if $b\in [0,1/2)$ and at $x=-\sqrt{\pi}/2$ if $b\in [1/2,1)$. This corresponds to Lemma \ref{fguv} with $u=2$ and $v=1$.
Examples of \e{rrx} for $s=3/4+400 i$ and  different values of $N$  are displayed in Table \ref{tb2:1}, correct to the accuracy shown.

\begin{table}[ht]
\centering
\begin{tabular}{ccc}
\hline
$N$ & Theorem \ref{unsym-rs} & \\
\hline
$1$ & $0.11\textcolor{gray}{628656704} + 0.031\textcolor{gray}{02038722} i$ & \\
$3$ & $0.11503\textcolor{gray}{659264} + 0.031341\textcolor{gray}{63666} i$ & \\
$5$ & $0.11503572\textcolor{gray}{670} + 0.03134146\textcolor{gray}{229} i$ & \\
\hline
 & $0.11503572550 + 0.03134146183 i$ & $R$\\
\hline
\end{tabular}
\caption{The approximations of Theorem \ref{unsym-rs} to $R=R(3/4+400 i;20/\sqrt{\pi},10/\sqrt{\pi})$.} \label{tb2:1}
\end{table}

\subsection{An example with increasing $\lambda$}
Suppose we take $\alpha=t^c$ in Theorem \ref{unsym-rs} for some $c>1/2$. Then
\begin{equation*}
  \beta=\frac{t}{2\pi \alpha} =\frac{t^{1-c}}{2\pi}, \qquad \lambda =\sqrt{\frac \alpha \beta}=\sqrt{2\pi} \, t^{c-1/2}.
\end{equation*}
The error term in \e{rrss} is $O(t^{-(c-1/2)(\sigma-1/2)+N(3c-2)-1/2})$ and so we require $c<2/3$ for this to decrease with $N$. If we take $c=5/8$, for example, then Theorem \ref{unsym-rs} gives
\begin{multline} \label{ssq}
  R(s;t^{5/8},t^{3/8}/(2\pi)) =
   e^{\pi i A(t)}  (2\pi)^{3/4-s}t^{-1/8-s/4}
   \\
      \times      \sum_{n=0}^{N-1} \frac {1}{t^{n/2}} \left[ \sum_{r=0}^{3n}
    \frac{G^{(r)}(\frac a{\sqrt{2\pi}} t^{-1/8} +b \sqrt{2\pi} t^{1/8} ;2\pi t^{1/4} )}{(2\pi)^{r/2}} \cdot P_{n,3n-r}\left(\frac a2 t^{-1/8} -b \pi t^{1/8},\sigma\right)\right]
\\
+O\left( \frac {1}{t^{(N+1+\sigma)/8}}\right)
\end{multline}
for
\begin{equation*}
  A(t):=\lfloor t^{5/8} \rfloor \lfloor t^{3/8}/(2\pi) \rfloor+1+ a t^{5/8}- b\frac{t^{3/8}}{2\pi}+\frac{a^2}{4\pi} t^{-1/4}-b^2 \pi t^{1/4}
\end{equation*}
and $a$, $b$ the fractional parts of $t^{5/8}$, $t^{3/8}/(2\pi)$. The derivatives of $G$ in \e{ssq} may be expressed in terms of Hermite polynomials and derivatives of $\Upsilon$ as in \e{fwy}. Then the derivatives of $\Upsilon$ can be computed with \e{updiff}. Table \ref{mich} shows each side of \e{ssq} when $s=1/2+256 i$.

\begin{table}[ht]
\centering
\begin{tabular}{ccc}
\hline
$N$ & Theorem \ref{unsym-rs} & \\
\hline
$1$ & $-0.12\textcolor{gray}{120812956} + 0.00\textcolor{gray}{884587559} i$ & \\
$2$ & $-0.1207\textcolor{gray}{5592244} + 0.0078\textcolor{gray}{9494686} i$ & \\
$4$ & $-0.120742\textcolor{gray}{08191} + 0.0078772\textcolor{gray}{9724} i$ & \\
\hline
 & $-0.12074212743 + 0.00787728177  i$ & $R$\\
\hline
\end{tabular}
\caption{The approximations of Theorem \ref{unsym-rs} to $R=R(1/2+256 i;32,4/\pi)$.} \label{mich}
\end{table}

It is natural to consider the difference $R(s;\alpha,\beta)-R(s;\alpha',\beta')$, as the $\zeta(s)$ terms cancel. With $\alpha=t^{5/8}$ as above and $\alpha'=t^{5/8}/2$ we may obtain the asymptotics of
\begin{equation*}
  e^{i\vartheta(s)}\Biggl( \sum_{\frac{t^{5/8}}2 < n \lqs t^{5/8}} \frac {1}{n^{s}}\Biggr) -e^{i\vartheta(1-s)}\Biggl( \sum_{\frac{t^{3/8}}{2\pi} < n\lqs \frac{t^{3/8}}{\pi}} \frac {1}{n^{1-s}}\Biggr),
\end{equation*}
for example, with \e{ssq} minus a similar expression.

\subsection{On the line $\Re(s)=1$}
When $\sigma=1$ there are some simplifications in the definition of $P_{n,k}(x,\sigma)$ in \e{poly}. With basic properties of Bernoulli polynomials we find
\begin{equation*}
  f_n(1)=\frac{B_{n+1}}{n(n+1)2^{n+1}}.
\end{equation*}
Also, for $d_{m,r}(1)$ in \e{poly:c}, only the $k=0$ term can be non-zero. Then $\hat{\mathcal B}_{m-n,0}(1,-\frac 12,\frac 13,\dots)$ is zero unless $n=m$. Therefore
\begin{equation*}
  d_{m,r}(1) = \hat{\mathcal B}_{m,r}(\textstyle{\frac 13,-\frac 14,\frac 15,\dots})/r!.
\end{equation*}
With this, the first polynomials $P_{n,k}(x,1)$ for $0\lqs k\lqs 3n$ are $P_{0,0}(x,1)=1$ and
\begin{equation*}
  P_{1,0}(x,1)= -\textstyle{\frac 13}, \quad P_{1,1}(x,1)= -i x, \quad
P_{1,2}(x,1)= x^2 -\textstyle{\frac i2}, \quad P_{1,3}(x,1)= \textstyle{\frac i3}x^3 + \textstyle{\frac 12}x.
\end{equation*}
For $n=2$ we have
\begin{gather}
  P_{2,0}(x,1)= \textstyle{\frac 1{18}}, \quad P_{2,1}(x,1)= \textstyle{\frac i3} x, \quad
P_{2,2}(x,1)=  -\textstyle{\frac 56} x^2 +\textstyle{\frac i6}, \quad
 P_{2,3}(x,1)= -\textstyle{\frac {10i}9}x^3 - \textstyle{\frac 23}x, \notag\\
  P_{2,4}(x,1)= \textstyle{\frac 56}x^4 -i x^2+\textstyle{\frac 18}, \quad
P_{2,5}(x,1)= \textstyle{\frac i3}x^5 + \textstyle{\frac 23}x^3 + \textstyle{\frac i4}x, \notag\\
  P_{2,6}(x,1)= -\textstyle{\frac 1{18}}x^6 + \textstyle{\frac i6}x^4 - \textstyle{\frac 18}x^2 + \textstyle{\frac i8}. \label{p2rs}
\end{gather}
For example, taking $\sigma=1$, $t=600$ and $\alpha/\beta=5/3$ in Theorem \ref{unsym-rs} gives the results in Table \ref{tabre}.

\begin{table}[ht]
\centering
\begin{tabular}{ccc}
\hline
$N$ & Theorem \ref{unsym-rs} & \\
\hline
$1$ & $0.07\textcolor{gray}{827091811} - 0.076\textcolor{gray}{57008324} i$ & \\
$3$ & $0.07798\textcolor{gray}{494014} - 0.076932\textcolor{gray}{55693} i$ & \\
$5$ & $0.077985048\textcolor{gray}{83} - 0.0769326604\textcolor{gray}{7} i$ & \\
\hline
 & $0.07798504890 - 0.07693266040 i$ & $R$\\
\hline
\end{tabular}
\caption{The approximations of Theorem \ref{unsym-rs} to $R=R(1+600 i;\sqrt{500/\pi},\sqrt{180/\pi})$.} \label{tabre}
\end{table}

\subsection{On the critical line}
For $\sigma=1/2$ we have already seen with \e{npp3} that the polynomials $P_{n,k}(x,\sigma)$ take a simpler form; only coefficients of powers of $x$ that are congruent to $k \bmod 4$ can be non-zero. For example,
$P_{0,0}(x,1/2)=1$ and
\begin{equation} \label{beto2}
  P_{1,0}(x,1/2)= -\textstyle{\frac 13}, \quad P_{1,1}(x,1/2)= -i x, \quad
P_{1,2}(x,1/2)= x^2 , \quad P_{1,3}(x,1/2)= \textstyle{\frac i3}x^3.
\end{equation}
Of course, \e{beto2} is a special case of \e{beto}. For $n=2$ we have
\begin{gather}
  P_{2,0}(x,1/2)= \textstyle{\frac 1{18}}, \quad P_{2,1}(x,1/2)= \textstyle{\frac i3} x, \quad
P_{2,2}(x,1/2)=  -\textstyle{\frac 56} x^2, \quad
 P_{2,3}(x,1/2)= -\textstyle{\frac {10i}9}x^3 , \notag\\
  P_{2,4}(x,1/2)= \textstyle{\frac 56}x^4 +\textstyle{\frac 14}, \quad
P_{2,5}(x,1/2)= \textstyle{\frac i3}x^5 +  \textstyle{\frac i2}x, \quad 
  P_{2,6}(x,1/2)= -\textstyle{\frac 1{18}}x^6  - \textstyle{\frac 14}x^2. \label{p2rsb}
\end{gather}
We give a more detailed  numerical example in Table \ref{jeb} for $s=1/2+800i$ and $\alpha/\beta=4$.

\begin{table}[ht]
\centering
\begin{tabular}{ccc}
\hline
$N$ & Theorem \ref{unsym-rs} & \\
\hline
$1$ & $-0.079\textcolor{gray}{66764263636} - 0.073\textcolor{gray}{73504930114} i$ & \\
$3$ & $-0.079573\textcolor{gray}{71736089} - 0.07351\textcolor{gray}{910859701}  i$ & \\
$5$ & $-0.079573651\textcolor{gray}{82034} - 0.073518978\textcolor{gray}{39664} i$ & \\
$7$ & $-0.0795736517815\textcolor{gray}{8} - 0.073518978259\textcolor{gray}{65} i$  &  \\
\hline
 & $-0.07957365178152 - 0.07351897825948 i$ & $R$\\
\hline
\end{tabular}
\caption{The approximations of Theorem \ref{unsym-rs} to $R=R(1/2+800 i;40/\sqrt{\pi},10/\sqrt{\pi})$.} \label{jeb}
\end{table}

Specializing \e{polyc} to $\sigma=1/2$, and extending the calculation to $x^{k-12}$, shows that the highest degree terms in $P_{n,k}(x,1/2)$ are
\begin{multline} \label{highp}
  \frac{P_{n,k}(x,1/2)}{(-1)^n i^k} = \binom{3n}{k}\frac{x^k}{3^{n}n!}+  \binom{3n-4}{k-4} \frac{3n-1}{20}\frac{x^{k-4}}{3^{n-2}(n-2)!}\\
  +\binom{3n-8}{k-8}
  \frac{63n^2-141n+31}{5600}(3n-5)\frac{x^{k-8}}{3^{n-3}(n-3)!}\\
+\binom{3n-12}{k-12}
  \frac{567 n^4 - 4374 n^3 + 10968 n^2 - 9621 n +1280}{112000 }(n-3)\frac{x^{k-12}}{3^{n-4}(n-4)!}+ \cdots .
\end{multline}
The formulas in \e{highp} only make sense for $n$ large enough. For $P_{n,k}(x,1/2)$ to contain $x^{k-12}$ for instance, we need $k \gqs 12$ and hence $n \gqs 4$. It may be verified that the coefficient of $x^{k-12}$ in \e{highp} is always positive for these $n$ and $k$. Similarly, the higher powers of $x$ in \e{highp} always have positive coefficients. Combining these calculations with \e{npp3}, we have proved:
\begin{prop}
The terms in  $P_{n,k}(x,1/2)/((-1)^n i^k)$ may only contain powers of $x$ of the form $x^{k-4m}$ for $0\lqs m \lqs k/4$. For these $m$ values,
the coefficients of  $x^{k-4m}$  are always positive if $0\lqs m \lqs 3$.
\end{prop}

To examine this positivity further, let $S_{n,k}$ be the set of
 all the coefficients of $x^{k-4m}$ for $0\lqs m \lqs k/4$ in $P_{n,k}(x,1/2)/((-1)^n i^k)$. Let $S_n$ be the union of these $S_{n,k}$ for $k$ in the range $0\lqs k\lqs 3n$. Then further computations show that all the elements of $S_n$ are positive for  $0\lqs n \lqs 50$. It seems likely that this positivity continues for all $n$. This would also imply that the sign pattern for $C_n(a)$ we see in \e{riem} continues for all $n$, with positive coefficients for $n$ even and negative coefficients for $n$ odd. The signs of the coefficients of $P_{n,k}(x,\sigma)$ also appear to obey predictable patterns, at least for $\sigma$ not too far from $1/2$.

\SpecialCoor
\psset{griddots=5,subgriddiv=0,gridlabels=0pt}
\psset{xunit=0.7cm, yunit=0.7cm, runit=0.7cm}
\psset{linewidth=1pt}
\psset{dotsize=3.5pt 0,dotstyle=*}
\psset{arrowscale=1.5,arrowinset=0.3}
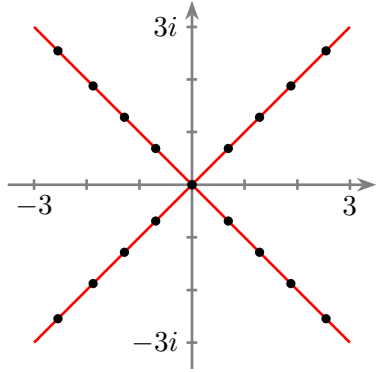
\begin{figure}[ht]
\centering
\begin{pspicture}(-3.5,-3.5)(3.5,3.5) 

\psline[linecolor=gray]{->}(-3.5,0)(3.5,0)
\psline[linecolor=gray]{->}(0,-3.5)(0,3.5)

\psline[linecolor=red](-3,-3)(3,3)
\psline[linecolor=red](-3,3)(3,-3)

\multirput(-3,-0.1)(1,0){7}{\psline[linecolor=gray](0,0)(0,0.2)}
\multirput(-0.1,-3)(0,1){7}{\psline[linecolor=gray](0,0)(0.2,0)}

\savedata{\mydatb}[
{{-2.54565, -2.54565}, {-2.54565, 2.54565}, {-1.87821,
  1.87821}, {-1.87821, -1.87821}, {-1.28299, -1.28299}, {-1.28299,
  1.28299}, {-0.689671, -0.689671}, {-0.689671, 0.689671}, {0.,
  0}, {0., 0}, {0.689671, 0.689671}, {0.689671, -0.689671}, {1.28299,
  1.28299}, {1.28299, -1.28299}, {1.87821, -1.87821}, {1.87821,
  1.87821}, {2.54565, -2.54565}, {2.54565, 2.54565}}
]
\dataplot[plotstyle=dots]{\mydatb}


\rput(3,-0.4){$3$}
\rput(-3,-0.4){$-3$}
\rput(-0.5,3){$3i$}
\rput(-0.7,-3){$-3i$}

\end{pspicture}
\caption{The zeros of $P_{6,18}(x,1/2)$}
\label{zeros}
\end{figure}

Another interesting aspect of $P_{n,k}(x,1/2)$ is that, in all the cases we have examined, its zeros are on the lines bisecting the quadrants and are nearly evenly spaced. Figure \ref{zeros} shows the zeros of
\begin{equation*}
  P_{6,18}(x,1/2) = - \textstyle{\frac{1}{524880}}x^{18} - \textstyle{\frac{17}{38880}}
 x^{14} - \textstyle{\frac{18889}{907200}} x^{10} - \textstyle{\frac{367}{1920}} x^6 -\textstyle{\frac{5}{32}} x^2.
\end{equation*}
For $\sigma$ near $1/2$ the zeros of $P_{n,k}(x,\sigma)$ appear to have a similar distribution. 

\subsection{Conclusion}
We have shown that the Riemann-Siegel formula  and the Hardy-Littlewood approximate functional equation are special cases of a   shared natural generalization in Theorem \ref{unsym-rs}.  The classical Riemann-Siegel coefficients $C_n(a)$ are given in our reformulation by
\begin{equation} \label{rty}
   C_n(a) = \sum_{r=0}^{3n}
    \frac{G^{(r)}(2a ;1 )}{(2\pi)^{r/2}} \cdot P_{n,3n-r}\left(0,1/2\right)
\end{equation}
as is seen by comparing \e{rt} with \e{rrss40}.
In the wider context of  Theorem \ref{unsym-rs}, we need the more general Mordell integral $G(u;\tau)$, and the constant terms $P_{n,k}\left(0,1/2\right)$ in \e{rty} are replaced with the  polynomials $P_{n,k}(x,\sigma)$ in $x$ and $\sigma$. The  fascinating properties of Mordell integrals have attracted many authors, as we have seen in Sections \ref{aup}, \ref{bup}.  The key symmetry of $G(u;\tau)$ as $\tau \to 1/\tau$ is related through Theorem \ref{unsym-rs} to the functional equation of $\zeta(s)$.  The polynomials $P_{n,k}(x,\sigma)$ inherit a functional equation from $\zeta(s)$, (Theorem \ref{polss}), and as noted above they also  seem to inherit interesting zeros.

In  future work we will examine these components $G(u;\tau)$ and $P_{n,k}(x,\sigma)$ in greater detail.
Also a natural  extension of the techniques in this paper is to Dirichlet $L$-functions $L(s,\chi)$. This would generalize the treatments in \cite{Si43} and \cite{De67}.

{\small
\bibliography{rs-bib}
}

{\small 
\vskip 5mm
\noindent
\textsc{Dept. of Math, The CUNY Graduate Center, 365 Fifth Avenue, New York, NY 10016-4309, U.S.A.}

\noindent
{\em E-mail address:} \texttt{cosullivan@gc.cuny.edu}
}

\end{document}